\documentclass[12pt, leqno,twoside]{article}
\usepackage{amssymb}
\usepackage{amsmath}
\numberwithin{equation}{section}
\usepackage{amsthm}
\usepackage{graphicx}
\usepackage[pagebackref,colorlinks,citecolor=blue,linkcolor=blue]{hyperref}
\usepackage{cite}

\usepackage{todonotes}
\usepackage{amsmath}
\usepackage{amsfonts}
\usepackage{amssymb}
\usepackage{amsmath,bbm,amssymb,amsxtra}
\usepackage{mathrsfs}
\usepackage{enumerate}
\usepackage{caption}
\allowdisplaybreaks
\usepackage{epsfig}
\usepackage{ae}
\usepackage{epstopdf}
\def\vint{\mathop{\mathchoice%
         {\setbox0\hbox{$\displaystyle\intop$}\kern 0.22\wd0%
          \vcenter{\hrule width 0.6\wd0}\kern -0.82\wd0}%
         {\setbox0\hbox{$\textstyle\intop$}\kern 0.2\wd0%
          \vcenter{\hrule width 0.6\wd0}\kern -0.8\wd0}%
         {\setbox0\hbox{$\scriptstyle\intop$}\kern 0.2\wd0%
          \vcenter{\hrule width 0.6\wd0}\kern -0.8\wd0}%
         {\setbox0\hbox{$\scriptscriptstyle\intop$}\kern 0.2\wd0%
          \vcenter{\hrule width 0.6\wd0}\kern -0.8\wd0}}%
         \mathopen{}\int}

\newcommand{\diam}{\text{\rm\,diam}}

\newcommand{\dist}{\text{\rm dist}}

\newtheorem{thm}{Theorem}[section]

\newtheorem{cor}[thm]{Corollary}
\newtheorem{lem}[thm]{Lemma}
\newtheorem{prop}[thm]{Proposition}

\newtheorem{claim}{Claim}[section]
\newtheorem{subclaim}{Subclaim}
\newtheorem{conj}[equation]{Conjecture}
\newtheorem{case}{Case}[section]
\newtheorem*{mysolution}{Solution}
\newtheorem{step}{Step}[section]
\theoremstyle{definition}
\newtheorem{defn}[thm]{Definition}

\newtheorem{prob}[equation]{Problem}
\newtheorem{ques}[equation]{Question}
\newtheorem{rem}{Remark}[section]
\newtheorem{rems}{Remarks}[section]
\newcounter {own}
\def\theown {\thesection       .\arabic{own}}

\newenvironment{pf}[1][]{%
 \vskip 3mm
 \noindent
 \ifthenelse{\equal{#1}{}}%
  {{\slshape Proof. }}%
  {{\slshape #1.} }%
 }%
{\qed\bigskip}

\newcounter{alphabet}

\makeatletter

\newcommand{\IR}{{\mathbb R}}

\newcommand{\RNum}[1]{\uppercase\expandafter{\romannumeral #1\relax}}

\def\be{\begin{equation}}
\def\ee{\end{equation}}

\newcommand{\ben}{\begin{enumerate}}
\newcommand{\een}{\end{enumerate}}

\newcommand{\blem}{\begin{lem}}
\newcommand{\elem}{\end{lem}}
\newcommand{\bthm}{\begin{thm}}
\newcommand{\ethm}{\end{thm}}
\newcommand{\bcor}{\begin{cor}}
\newcommand{\ecor}{\end{cor}}
\newcommand{\beg}{\begin{examp}}
\newcommand{\eeg}{\end{examp}}
\newcommand{\begs}{\begin{examples}}
\newcommand{\eegs}{\end{examples}}
\newcommand{\bdefe}{\begin{defn}}
\newcommand{\edefe}{\end{defn}}
\newcommand{\bprob}{\begin{prob}}
\newcommand{\eprob}{\end{prob}}
\newcommand{\bques}{\begin{ques}}
\newcommand{\eques}{\end{ques}}
\newcommand{\bei}{\begin{itemize}}
\newcommand{\eei}{\end{itemize}}
\newcommand{\bcl}{\begin{claim}}
\newcommand{\ecl}{\end{claim}}
\newcommand{\bscl}{\begin{subclaim}}
\newcommand{\escl}{\end{subclaim}}
\newcommand{\bca}{\begin{case}}
\newcommand{\eca}{\end{case}}
\newcommand{\bstep}{\begin{step}}
\newcommand{\estep}{\end{step}}
\newcommand{\bsol}{\begin{mysolution}}
\newcommand{\esol}{\end{mysolution}}
\newcommand{\bcon}{\begin{conj}}
\newcommand{\econ}{\end{conj}}
\newcommand{\bcons}{\begin{conjs}}
\newcommand{\econs}{\end{conjs}}
\newcommand{\bprop}{\begin{prop}}
\newcommand{\eprop}{\end{prop}}
\newcommand{\br}{\begin{rem}}
\newcommand{\er}{\end{rem}}
\newcommand{\brs}{\begin{rems}}
\newcommand{\ers}{\end{rems}}
\newcommand{\bo}{\begin{obser}}
\newcommand{\eo}{\end{obser}}
\newcommand{\bos}{\begin{obsers}}
\newcommand{\eos}{\end{obsers}}

\newcommand{\bpf}{\begin{pf}}
\newcommand{\epf}{\end{pf}}

\newcommand{\ba}{\begin{array}}
\newcommand{\ea}{\end{array}}
\newcommand{\beq}{\begin{eqnarray}}
\newcommand{\beqq}{\begin{eqnarray*}}
\newcommand{\eeq}{\end{eqnarray}}
\newcommand{\eeqq}{\end{eqnarray*}}

\begin{document}

\title{\Large\bf Quasi-symmetries between metric spaces and rough quasi-isometries between their infinite hyperbolic cones
 \footnotetext{\hspace{-0.35cm}
 $2020$ {\it Mathematics Subject classfication}: 30C65, 53C23
 \endgraf{{\it Key words and phases}: quasi-symmetry, rough quasi-isometry, (infinite) hyperbolic cone, Gromov hyperbolic space.}
}
}
\author{Manzi Huang and Zhihao Xu}
\date{ }
\maketitle

\begin{abstract}
In this paper, we first prove that any power quasi-symmetry of two metric spaces induces a rough quasi-isometry between their infinite hyperbolic cones. Second,  we prove that for a complete metric space $Z$, there exists a point $\omega$ in the Gromov boundary of its infinite hyperbolic cone such that $Z$ can be seen as the Gromov boundary relative to $\omega$ of its infinite hyperbolic cone.
Third, we prove that for a visual Gromov hyperbolic metric space $X$ and a Gromov boundary point $\omega$, $X$ is roughly similar to the infinite hyperbolic cone of its Gromov boundary relative to $\omega$.
These are the generalizations of Theorem 7.4, Theorem 8.1 and Theorem 8.2 in \cite{BSC} since the underlying spaces are not assumed to be bounded and the hyperbolic cones are infinite.
\end{abstract}

\section{Introduction}\label{sec-1}

Gromov hyperbolic spaces have been studied by many authors, see e.g., \cite{V, BuSc, BSC, BHK, GH, BuC, Gr87, BBS}.  Among the authors,
Bonk-Schramm proved that every Gromov hyperbolic geodesic metric space with {\it bounded growth at some scale} is roughly similar to a convex subset of a hyperbolic $n$-space $\mathbb H^n$ for some integer $n$ (see \cite[Theorem 10.2]{BSC}). See \cite[Section 9]{BSC} for the definition of bounded growth at some scale.
 To prove this result, Bonk-Schramm constructed a class of Gromov hyperbolic spaces based on bounded metric spaces $(Z, d_Z)$, named as {\it hyperbolic cones}, which are defined as follows:
\begin{equation*}
	\text{Con}(Z)=Z\times(0, D(Z)],
\end{equation*}
where $D(Z)= \diam Z$. Then they demonstrated that $(\text{Con}(Z),\rho)$ is Gromov hyperbolic (see \cite[Theorem 7.2]{BSC}), where the metric $\rho:\text{Con}(Z)\times\text{Con}(Z)\to[0,\infty)$
is defined by
\begin{equation}\label{metric-cone}
	\rho((x,s), (y,t))=2\log\bigg(\frac{d_Z(x, y)+s\vee t}{\sqrt{st}}\bigg)
\end{equation}
(see \cite[Lemma 7.1]{BSC}).

In \cite{BSC}, the authors asked what can remain of the conclusion of \cite[Theorem 10.2]{BSC} if the assumption that $X$ has bounded growth at some scale is removed. Bonk-Schramm considered this problem themselves in \cite{BSC}. To discuss this problem, they introduced a class of metric spaces $(\text{Con}_h(Z), d_h)$, where
\beq\label{11-16-1}
	\text{Con}_h(Z)=Z\times(0, \infty),
\eeq
 the metric $d_{h}:$ $\text{Con}_h(Z)\times\text{Con}_h(Z)\to[0,\infty)$
is defined by the formula:
\begin{equation*}
d_{h}((x, s), (y, t))={\rm cosh }^{-1}\Big(1+\frac{d_Z(x, y)^2+(s-t)^2}{2st}\Big)
\end{equation*} for $(x, s)$ and $(y, t)\in\text{Con}_h(Z)$,
and ${\rm cosh}$ denotes the hyperbolic cosine function.
Then all $(\text{Con}_h(Z), d_{h})$ are Gromov hyperbolic (see \cite[Section 10]{BSC}), and
every Gromov hyperbolic space is roughly similar into the space $(\text{Con}_h(l_\infty(A)), d_h)$ (see \cite[Theorem 10.4]{BSC} for the details).
See \cite{Gr87, TrV} for more discussions on this line.

Obviously, we can extend the metric defined by \eqref{metric-cone} from the hyperbolic cone $\text{Con}(Z)$ to its infinite case $\text{Con}_h(Z)$, and in the following, we denote the extended metric by $\rho_h$. Sometimes, we also use the notation $\rho_{h,Z}$ to emphasize the underlying spaces.
Then the similar reasoning as in the proofs of \cite[Lemma 7.1 and Theorem 7.2]{BSC} ensures that $\rho_h$ is a metric on $\text{Con}_h(Z)$, and the metric space $(\text{Con}_h(Z), \rho_h)$ is Gromov hyperbolic.
 In this paper, we always equip ${\rm Con}_h (Z)$ with the metric $\rho_h$. To distinguish ${\rm Con}_h(Z)$ from ${\rm Con}(Z)$, in the following, we call it the {\it infinite hyperbolic cone} of $Z$.

 In \cite{BSC}, Bonk-Schramm proved that, for any power quasi-symmetry between two bounded metric spaces $(Z,d_Z)$ and $(W,d_W)$, there is a rough quasi-isometry between ${\rm Con}(Z)$ and ${\rm Con}(W)$  (see \cite[Theorem 7.4]{BSC}).
The first purpose of this paper is to establish an analogue of \cite[Theorem 7.4]{BSC} in our setting, that is, the underlying metric spaces are not assumed to be bounded and the hyperbolic cones are infinite. Since this study concerns quasi-symmetries, the metric spaces involved in this paper are assumed to have at least three points.
Our first result is as follows.

\begin{thm}\label{thm-1}
	Suppose that $f$: $(Z, d_Z)\to (W, d_W)$ is a $(\theta, \lambda)$-power quasi-symmetry with $\theta\geq 1$ and $\lambda\geq 1$. Then there is a $(\theta, k)$-rough quasi-isometry $\widehat{f}$:  $({\rm Con}_h(Z),\rho_{h,Z}) \to ({\rm Con}_h(W),\rho_{h,W})$, where $k=k(\theta, \lambda)$. Furthermore, if $f$ is an $(\alpha, C)$-snowflake mapping with $\alpha>0$ and $C\geq 1$, then $\widehat{f}$ is an $(\alpha,k^\prime)$-rough
	similarity with $k^\prime=k^\prime(\alpha, C)$.
\end{thm}

Here, the notation $k(\theta, \lambda)$ denotes a positive constant depending only on $\theta$ and $\lambda$.

For a Gromov hyperbolic space $X$, we use the notation $\partial_{G}X$ to denote the {\it Gromov boundary} of $X$ defined as the set of all equivalence classes of sequences converging to infinity.
Given a point $\omega\in \partial_G X$, we use the notation $\partial_\omega X$ to denote the {\it Gromov boundary relative to $\omega$} of $X$ defined as the set of all equivalence classes of sequences in $X$ converging to infinity with respect to $\omega$. The precise definitions of the boundaries $\partial_G X$ and $\partial_\omega X$ will be given in Section \ref{sec-3-1} below. It is known that there is an identification of $\partial_\omega X$ and $\partial_G X\setminus\{\omega\}$, see Proposition \ref{prop-Gromov} below.

In \cite{BSC}, Bonk-Schramm also proved that, as sets, the Gromov boundary $\partial_G{\rm Con}(Z)$ of ${\rm Con}(Z)$ and $Z$ can be identified with each other, and $d_Z$ induces a visual metric on $\partial_G{\rm Con}(Z)$ with parameter $1$, when $(Z, d_Z)$ is a complete bounded metric space.
 Furthermore, if $X$ is a visual Gromov hyperbolic metric space, then $X$ and ${\rm Con}(\partial_G X)$ are roughly similar (see \cite[Theorem 8.1 and Theorem 8.2]{BSC}).
Recall that a metric space $X$ is called $k$-{\it visual with respect to} $o\in X$, if each point in $X$ lies on a $k$-roughly geodesic ray emanating from $o$. We call $X$ {\it visual}, if it is $k$-visual with respect to some point $o\in X$ for some $k\geq 0$.
The second purpose of this paper is to establish the analogues of \cite[Theorem 8.1]{BSC} and \cite[Theorem 8.2]{BSC} by using the Gromov boundary relative to $\omega$ to replace the Gromov boundary.

For a complete metric space $(Z, d_Z)$, there is a point $\omega\in \partial_G {\rm Con}_h(Z)$ such that any sequence $\{(z_n, h_n)\}$ in ${\rm Con}_h(Z)$ which satisfies
 $
\lim_{n\to\infty}h_n=\infty
$
belongs to the equivalence class of $\omega$, see Lemma \ref{union} below. Then we get the Gromov boundary relative to $\omega$ of ${\rm Con}_h(Z)$, denoted by $\partial_\omega {\rm Con}_h(Z)$.

\begin{thm}\label{thm-2}
Suppose that $(Z, d_Z)$ is a complete metric space. Then $\partial_\omega {\rm Con}_h(Z)$ and $Z$ can be identified as sets, and $d_Z$ induces a visual metric $d$ on $\partial_\omega {\rm Con}_h(Z)$ based at $\omega$ with parameter $1$.
 \end{thm}

Remark \ref{rem-3-1} below shows that Theorem \ref{thm-2} coincides with \cite[Theorem 8.1]{BSC} when $(Z, d_Z)$ is a complete bounded metric space.

In the study of asymptotic geometry, Buyalo-Schroeder constructed a class of Gromov hyperbolic spaces, named as {\it hyperbolic fillings} \cite{BuSc}.
A hyperbolic filling of a metric space $(Z, d_Z)$ is a simplicial graph with construction parameters $\alpha$ and $\tau$.
See Section \ref{sec-3} for the details.
Given a completed metric space $(Z, d_Z)$, the hyperbolic cone ${\rm Con}_h(Z)$ is roughly similar to every hyperbolic filling of $Z$, see Theorem \ref{thm-2.9} below.
Jordi proved that for a bounded metric space $(Z, d_Z)$, each hyperbolic filling of $(Z, d_Z)$ is roughly isometric to the hyperbolic filling of $(Z, d_Z^\prime)$ with same construction parameters, where $d_Z^\prime$ is the inversion of $d_Z$ at a point in $Z$ (see \cite[Theorem 7]{Jordi}).
Based on these facts, we get the following result which is a generalization of \cite[Theorem 8.2]{BSC}.

\begin{thm}\label{thm-3}
Suppose that $X$ is a visual Gromov hyperbolic metric space and $\omega\in \partial_G X$. Then $X$ and ${\rm Con}_h(\partial_\omega X)$ are roughly similar.
 \end{thm}

Throughout this paper, the letter $C = C(a,b,\ldots)$ will denote a positive constant that depends only on $a,$ $b,$ $\ldots$, and may change at different occurrences.
The notation $O(1)$ will be used for a real number whose absolute value is bounded by a positive numerical constant that can be computed explicitly in principle. Similarly, we use the notation $O_{a, b, c,\cdots}(1)$ for a real number whose absolute value is bounded by a constant which can be chosen in a way only depending on the parameters $a, b, c,\cdots$.

The paper is organized as follows. In Section \ref{sec-2}, necessary notions and notations will be introduced. In Section \ref{sec-3}, we will introduce infinite hyperbolic cones and hyperbolic fillings and discuss their intrinsic relationship.
Sections \ref{extension}, \ref{sec-5} and \ref{sec-6} shall be devoted to the proofs of Theorems \ref{thm-1}, \ref{thm-2} and \ref{thm-3}, respectively.

\section{Preliminaries}\label{sec-2}

In this section, we shall introduce necessary notions and notations.

\subsection{Rough quasi-isometries and quasi-symmetries}

	Let $(X, d_X)$ and $(Y, d_Y)$ denote metric spaces. The distance of two sets $A, B\subset X$ is denoted by $\dist(A, B)$, i.e., $\dist(A, B)=\inf\{d_X(x,y):\; x\in A,\; y\in B\}$.
A set $A\subset X$ is called {\it $k$-cobounded} (in $X$) for $k\geq 0$ if $\dist(\{x\}, A)\leq k$ for every point $x\in X$. If $A$ is $k$-cobounded for some $k\geq0$, then we say that $A$ is {\it cobounded}.

Let $f$: $(X, d_X)\to (Y, d_Y)$ be a mapping (not necessary continuous), and let $k\geq 0$ and $\alpha\geq 1$ be constants. Suppose that $f(X)$ is $k$-cobounded in $Y$. If, in addition, for all $x, z\in X$,
\begin{equation*}
\alpha^{-1}d_X(x,z)-k\leq d_Y(f(x),f(z))\leq\alpha d_X(x,z)+k,
\end{equation*}
then $f$ is called an $(\alpha, k)$-{\it rough quasi-isometry}. If
\begin{equation}\label{similar}
\alpha d_X(x,z)-k\leq d_Y(f(x),f(z))\leq\alpha d_X(x,z)+k,
\end{equation}
then $f$ is an $(\alpha, k)$-{\it rough similarity}. If $\alpha=1$ in \eqref{similar}, then  $f$ is a {\it $k$-rough isometry}.

A {\it geodesic}, a {\it geodesic ray} or a {\it geodesic segment} in $X$ is an isometry $\gamma: I\to X$ into $X$, where $I$ is $\mathbb R$, $[0,\infty)$ or a closed segment in $\mathbb R$, respectively.
Similarly, if $\gamma: I\to X$ is a $k$-rough isometry of $I=\mathbb R$ into $X$, we call it
a {\it $k$-rough geodesic}. We also speak of {\it $k$-roughly geodesic rays} or {\it $k$-roughly
geodesic segments}, if $I=[0,\infty)$ or $I$ is a closed segment in $\mathbb R$, respectively.
A {\it geodesic metric space} is a metric space $X$ such that for any two points $x, y\in X$ there is a geodesic segment joining $x$ and $y$.

\begin{defn}\label{def-qs}
 A  homeomorphism $f$ between two metric spaces $(Z, d_Z)$ and $(W, d_W)$ is called {\it $\eta$-quasi-symmetric} if there exists a self-homeomorphism $\eta$ of $[0, +\infty)$ such that for all triples of points $x, y, z\in Z$,
	$$\frac{d_W(f(x), f(z))}{d_W(f(y), f(z))}\leq \eta\left(\frac{d_Z(x, z)}{d_Z(y, z)}\right).$$
	If there are constants $\theta\geq 1$ and $\lambda\geq 1$ such that
	\begin{equation*}\label{eq-1.1}
		\eta_{\theta, \lambda}(t)=
		\left\{\begin{array}{cl}
			\lambda t^{\frac{1}{\theta}}& \text{for} \;\; 0<t<1, \\
			\lambda t^{\theta}& \text{for} \;\; t\geq 1,
		\end{array}\right.
	\end{equation*}
	then $f$ is called a $(\theta, \lambda)$-{\it power quasi-symmetry}. Here, the notation $\eta_{\theta, \lambda}$ means that the control function $\eta$ depends only on the given parameters $\theta$ and $\lambda$.
	
A  homeomorphism $f$: $(Z, d_Z)\to (W, d_W)$ is called an {\it $(\alpha,C)$-snowflake mapping} if there exist constants $\alpha>0$ and $C\geq 1$ such that for all $x, y\in Z$,
\begin{equation*}
C^{-1}d_Z(x, y)^{\alpha}\leq d_W(f(x), f(y))\leq Cd_Z(x, y)^{\alpha}.
\end{equation*}
\end{defn}

Obviously, every snowflake mapping is a power quasi-symmetry. The inverses and the compositions of power quasi-symmetries are again power quasi-symmetries. Indeed, the inverse of a $(\theta, \lambda)$-power quasi-symmetry is a $(\theta, C)$-power quasi-symmetry, where $C=C(\theta, \lambda)$ (cf. \cite[Proposition 10.6]{H}).

A {\it quasimetric space} is a set $Z$ together with a map $d:Z\times Z\to[0,\infty)$ such that $(1)$ $d(x, y)=0$ if and only if $x=y$; $(2)$ $d(x, y)=d(y, x)$ for all $x$, $y\in Z$; $(3)$ there is a constant $C\geq 1$ such that $d(x,y)\leq C\max\{d(x, z), d(z, y)\}$ for all $x, y, z\in Z$. Obviously, each metric space is a quasimetric space.

For two quasimetric spaces $(Z, d_1)$ and $(Z, d_2)$, we say that $d_1$ is {\it biLipschitz equivalent} to $d_2$ if there exists a constant $C\geq 1$ such that for all $x, y\in Z$,
\begin{equation*}
C^{-1}d_1(x, y)\leq d_2(x, y)\leq Cd_1(x, y).
\end{equation*}

\subsection{Gromov hyperbolic spaces}\label{sec-3-1}

Let $(X, d_X)$ be a metric space.
let $x, y, o$ be three points in $X$. The {\it Gromov product} of $x$ and $y$ with respect to  $o$ is defined as
$$
(x|y)_{o}=\frac{1}{2}(d_X(x,o)+d_X(y,o)-d_X(x,y)).
$$
Then for all points $x, y, o, o^{\prime}\in X$, we have
\begin{equation}\label{eq-base-1}
\big|(x|y)_{o}-(x|y)_{o^{\prime}}\big|\leq d_X(o,o^{\prime}).
\end{equation}

A metric space $(X, d_X)$ is called {\it Gromov $\delta$-hyperbolic} for some $\delta\geq 0$ if for all points $x, y, z, o\in X$,
$$
(x|y)_{o}\geq(x|z)_{o}\wedge(z|y)_{o}-\delta.
$$
Here, $a\vee b$ $($resp. $a\wedge b$$)$ denote the maximum $($resp. the minimum$)$ of $a, b\in \overline{\mathbb{R}}=\mathbb{R}\cup \{\infty\}$.
Sometimes, we briefly say that $X$ is {\it Gromov hyperbolic}.

In the rest of this section, we assume that $X$ is a Gromov $\delta$-hyperbolic metric space for some $\delta\geq 0$ and assume that  $o\in X$ is a fixed point.

A sequence of points $\{x_{i}\}\subset X$ is said to {\it converge to infinity} if
$$
(x_{i}|x_{j})_{o}\to\infty \;\; \mbox{as}\;\; i, j\to\infty.
$$
Two sequences $\{x_{i}\}$ and $\{y_{i}\}$ that converge to infinity are said to be {\it equivalent} if
$$
(x_{i}|y_{i})_{o}\to\infty \;\; \mbox{as}\;\; i\to\infty.
$$
This defines an equivalence relation for sequences in $X$ converging to infinity.
The convergence to infinity of a sequence and the equivalence of two sequences do not depend on the choice of $o$ because of \eqref{eq-base-1}.

The {\it Gromov boundary} $\partial_{G}X$ of $X$ is defined as the set of all equivalence classes of sequences converging to infinity.  For a point  $\omega\in\partial_{G}X$ and a sequence $\{x_n\}$ converging to infinity, we say that $\{x_n\}$ {\it converges to} $\omega$ or $\{x_n\}\in \omega$, if $\{x_n\}$ belongs to the equivalence class of $\omega$.

Let $\xi\in \partial_{G}X$, and let $y\in X$. The Gromov product $(y|\xi)_o$  with respect to  $o$  is defined as follows:
\begin{equation}\label{eq-Gro}
(\xi|y)_o=(y|\xi)_o=\inf\left\{\liminf_{i\to\infty}(x_i|y)_o:\; \{x_i\}\in\xi\right\},
\end{equation}
and for $\zeta, \xi\in\partial_{G}X$, define
$$
(\zeta|\xi)_o=\inf\left\{\liminf_{i\to\infty}(x_i|y_i)_o:\; \{x_i\}\in\zeta\;\;\mbox{and}\;\; \{y_i\}\in\xi\right\}.
$$

For any $\varepsilon>0$, we define a quasimetric on $\partial_G X$ as follows: For any $\zeta, \xi\in\partial_G X$, define
\begin{equation}\label{v-o}
\vartheta_{\varepsilon, o}(\zeta, \xi)=e^{-\varepsilon(\zeta|\xi)_{o}}.
\end{equation}
A metric $d$ is a {\it visual metric} on $\partial_G X$ with parameter $\varepsilon$ if $d$ is biLipschitz equivalent to $\vartheta_{\varepsilon, o}$.
The notion of visual metrics based at $o$ does not depend on the choice of $o$ in $X$ by \eqref{eq-base-1}. Obviously, for any visual metrics $d_1$, $d_2$, the identity mapping ${\rm id}: (\partial_G X, d_1)\to (\partial_G X, d_2)$ is a snowflake mapping.

The boundary $\partial_G X$ equipped with any visual metric is complete and bounded, cf. \cite[Propsition 6.2]{BSC}.

For any $\omega\in \partial_G X$, define the {\it Busemann function} $b_{\omega, o}$:  $X\to \mathbb R$ by
\begin{equation}\label{Busemann}
b_{\omega, o}(x):=(\omega|o)_x-(\omega|x)_o,
\end{equation}
where $(\omega|o)_x$ and $(\omega|x)_o$ are Gromov products given in \eqref{eq-Gro}.


 \begin{defn}
 Let $\omega\in \partial_G X$. The set $\mathcal B(\omega)$ consists of all Busemann functions $b: X\to \mathbb R$ for which there are $o^\prime\in X$ and a constant $c\in\mathbb R$ with
$$
|b(x)-b_{\omega, o^\prime}(x)|\leq c, \ \ \forall \ x\in X.
$$
 \end{defn}

For any $b\in \mathcal{B}(\omega)$, define the {\it Gromov product $(x|y)_b$ based at $b$} for $x, y\in X$ by
$$
(x|y)_b=\frac{1}{2}(b(x)+b(y)-d_X(x,y)).
$$
For $\xi\in \partial_G X\setminus\{\omega\}$ and $y\in X$, define
$$
(\xi|y)_b=(y|\xi)_b=\inf\left\{\liminf_{i\rightarrow\infty}(x_i|y)_b:\; \{x_i\}\in\xi\right\},
$$
and for $\zeta, \xi\in \partial_G X\setminus\{\omega\}$, define
$$
(\zeta|\xi)_b=\inf\left\{\liminf_{i\rightarrow\infty}(x_i|y_i)_b:\; \{x_i\}\in\zeta\;\;\mbox{and}\;\; \{y_i\}\in\xi\right\}.
$$
It is known that given any Busemann functions $b_1$, $b_2\in \mathcal B(\omega)$,  there is a constant $c\in\mathbb R$ so that
\begin{equation}\label{differ-b}
|(x|y)_{b_1}-(x|y)_{b_2}|\leq c
\end{equation}
for all $x, y\in X\cup (\partial_G X\setminus\{\omega\})$ (cf. \cite[Section 3.1]{BuSc}).

Fix a function $b\in \mathcal{B}(\omega)$.  A sequence $\{x_n\}$ {\it converges to infinity with respect to $\omega$} if
$$
(x_m|x_n)_b\to  \infty \;\; \mbox{as}\;\; m, n\to  \infty,
$$
and two sequences $\{x_n\}$ and $\{y_n\}$ converging to infinity with respect to $\omega$ are {\it equivalent} if
$$
(x_n|y_n)_b\to  \infty \;\; \mbox{as}\;\;n\to  \infty.
$$
This defines an equivalence relation for sequences in $X$ converging to infinity with respect to $\omega$, see \cite[Section 3.4.1]{BuSc}.
The notations are both independent of the choice of the Busemann function $b$ based at $\omega$ because of \eqref{differ-b}.

The {\it Gromov boundary relative to} $\omega$ is defined as the set
$\partial_{\omega}X$ of all equivalence classes of sequences converging to infinity with respect to $\omega$. For $\zeta\in\partial_{\omega}X$ and a sequence $\{x_n\}$ that converges to infinity with respect to $\omega$, we say that $\{x_n\}$ {\it converges to $\zeta$ with respect to $\omega$}, if $\{x_n\}$ belongs to the equivalence class of $\zeta$.


\begin{lem}[{\cite[Lemma 3.2.4]{BuSc}}]\label{lem-fun-b}
Let $b\in \mathcal B(\omega)$. Then \ben

\item[$(1)$]
for any $\xi, \zeta\in\partial_G X\setminus\{\omega\}$ and any $\{x_n\}\in\xi$, $\{y_n\}\in\zeta$, we have
\begin{equation}\label{Lim}
(\zeta|\xi)_b\leq \liminf_{i\to \infty}(x_i|y_i)_b\leq \limsup_{i\to \infty}(x_i|y_i)_b\leq (\zeta|\xi)_b+600\delta,
\end{equation}
and the same holds if we replace $\zeta$ with $x\in X$.

\item[$(2)$]\label{lem-fun-b-2}
for any $\xi, \zeta, \eta\in X\cup (\partial_G X\setminus\{\omega\})$, we have
\begin{equation}\label{3-ponits}
(\xi|\zeta)_b\geq (\xi|\eta)_b \wedge (\eta|\zeta)_b-600\delta.
\end{equation}
\een
\end{lem}

\begin{prop}[{\cite[Proposition 3.4.1]{BuSc}}]\label{prop-Gromov}
A sequence $\{x_n\}$ converges to infinity with respect to $\omega$ if and only if $\{x_n\}$ converges to a point $\xi\in \partial_G X\setminus\{\omega\}$. This correspondence defines an identification of $\partial_\omega X$ and $\partial_G X\setminus\{\omega\}$.
\end{prop}

Under this identification, we shall thus use the notation $\partial_{\omega}X$ instead of $\partial_G X\setminus\{\omega\}$ throughout the rest of this paper.

For any $\varepsilon>0$ and any $b\in \mathcal B(\omega)$, define a function $\vartheta_{\varepsilon, b}$ on $\partial_{\omega}X$ by letting
\begin{equation}\label{visual-b}
\vartheta_{\varepsilon, b}(\zeta, \xi)=e^{-\varepsilon(\zeta|\xi)_{b}}, \ \ \  \forall \ \zeta, \xi\in\partial_{\omega}X.
\end{equation}
In general, $\vartheta_{\varepsilon, b}$ does not define a metric.
A metric $d$ is a {\it visual metric} on $\partial_{\omega}X$ based at $\omega$ with parameter $\varepsilon$ if
$d$ is biLipschitz equivalent to $\vartheta_{\varepsilon, b}$.
The notion of visual metrics on $\partial_{\omega}X$ does not depend on the choice of $b\in \mathcal B(\omega)$.

Using the similar argument with \cite[Propsition 6.2]{BSC}, we know that the boundary $\partial_\omega X$ equipped with any visual metric based at $\omega$ is complete.

\begin{lem}\label{lem-eq-b}
Let $\omega\in \partial_G X$, and let $b\in \mathcal B(\omega)$. Assume that $p, q \in X\cup \partial_\omega X$ and $x, y\in X$. If there is a constant $c\geq 0$ such that
\beq\label{fri-1}
(p|x)_b\geq b(x)-c\;\;\mbox{and}\;\;(q|y)_b\geq b(y)-c,
\eeq then
$$
d_X(x,y)=b(x)+b(y)-2(b(x)\wedge (p|q)_b \wedge b(y))+O_{\delta, c}(1).
$$
\end{lem}

\begin{proof}
Let $\omega\in \partial_G X$, and let $b\in \mathcal B(\omega)$.
It is known from \cite[Proposition 3.1.5]{BuSc} that for any $u, v\in X$,
$$|b(u)-b(v)|\leq d_X(u,v)+10\delta,$$
which shows that
\beq\label{thurs-1}
(u|v)_b\leq b(u)\wedge b(v)+5\delta.
\eeq

Let $p, q \in X\cup \partial_\omega X$ and $x, y\in X$. We infer from Lemma \ref{lem-fun-b}$(2)$ that
\begin{align*}
(p|q)_b &\geq (p|x)_b\wedge (x|q)_b-600\delta
             \geq (p|x)_b\wedge (x|y)_b \wedge (y|q)_b-1200\delta\\
              &\geq b(x)\wedge (x|y)_b \wedge b(y)-1200\delta-c,
\end{align*}
where, in the last inequality, the inequalities in \eqref{fri-1} are applied.
Then \eqref{thurs-1} gives
$$b(x)\wedge (p|q)_b \wedge b(y)\geq (x|y)_b-1205\delta-c.$$

Again, we deduce from  Lemma \ref{lem-fun-b}$(2)$ and \eqref{fri-1} that
\begin{align*}
(x|y)_b &\geq (x|p)_b\wedge (p|y)_b-600\delta
\geq (x|p)_b\wedge (p|q)_b \wedge (q|y)_b-1200\delta\\
&\geq  b(x)\wedge (p|q)_b \wedge b(y)-1200\delta-c.
\end{align*}
These imply that
$$
(x|y)_b = b(x)\wedge (p|q)_b \wedge b(y)+O_{\delta, c}(1),
$$
and so, we have
$$d_X(x,y)=b(x)+b(y)-2(b(x)\wedge (p|q)_b \wedge b(y))+O_{\delta, c}(1),$$ which is what we need.
\end{proof}

Assume that  $X$ contains a geodesic ray $\gamma: [0, \infty)\to X$. Observe that, there is a point $\omega\in \partial_G X$ such that all sequences $\{\gamma(t_n)\}$ with $t_n\to\infty$ belong to $\omega$.
We say $\omega$ the {\it endpoint} of $\gamma$.
 Define the function $b_\gamma$: $X\to \mathbb{R}$ associated to $\gamma$ by the limit
\begin{equation}\label{def-br}
b_{\gamma}(x)=\lim_{t\to \infty} \big(d_X(\gamma(t), x)-t\big).
\end{equation}
By \cite[Lemma 2.2.2 and Lemma 3.1.2]{BuSc}, for any $x\in X$,
$$
b_{\gamma}(x)=b_{\omega, \gamma(0)}(x)+O_\delta(1).
$$
Then $b_{\gamma}\in \mathcal{B}(\omega)$.

\section{Infinite hyperbolic cones and hyperbolic fillings}\label{sec-3}

\subsection{Infinite hyperbolic cones}

Let $(Z, d_Z)$ be a  metric space. As we know, $(\text{Con}_h(Z), \rho_h)$ is a Gromov $\delta$-hyperbolic metric space for some $\delta=O(1)\geq 0$, where $\text{Con}_h(Z)$ and $\rho_h$ are defined as in \eqref{11-16-1} and in \eqref{metric-cone}, respectively.
In the following, we give some properties of $\text{Con}_h(Z)$.

Define a mapping
\begin{equation*}
R_z: \mathbb R\to {\rm Con}_h(Z), \ \ t\mapsto (z, e^t).
\end{equation*}
Obviously, as a set in ${\rm Con}_h(Z)$, $R_z$ is given by
\begin{equation}\label{def-Rz}
	R_z=\{z\}\times(0, +\infty).
\end{equation}
A direct calcultion gives that
$$
\rho_h(R_z(t_1), R_z(t_2))=|t_2-t_1|.
$$
Then $R_z$: $\mathbb R\to {\rm Con}_h(Z)$ is an isometry, and so, $R_z$ is a geodesic in ${\rm Con}_h(Z)$.



\begin{lem}\label{union}
There exists $\omega\in \partial_G {\rm Con}_h(Z)$ such that for any sequence $\{x_n=(z_n, h_n)\}$ in ${\rm Con}_h(Z)$, if
$$
\lim_{n\to\infty}h_n=\infty,
$$
then $\{x_n\}\in \omega$.
\end{lem}

\begin{proof}
Fix a point $o=(z, h)\in {\rm Con}_h(Z)$. Let $x_n=(z_n, h_n)$ with
$$
\lim_{n\to\infty}h_n=\infty.
$$
Without loss of generality, assume that $h_n\geq h$ for all $n$. Then
\begin{align}\label{eq-23}
(x_n|x_m)_o =& \frac{1}{2}\big(\rho_h(x_n, o)+\rho_h(x_m, o)-\rho_h(x_n, x_m)\big)\notag\\
=& \log\frac{(d_Z(z_n, z)+h_n)(d_Z(z_m, z)+h_m)}{(d_Z(z_n, z_m)+h_n\vee h_m)h}.
\end{align}

In the following, we show that $(x_n|x_m)_o\to\infty$ as $m, n\to\infty$. It suffices to consider two cases by symmetry:  $d_Z(z_n, z)\geq d_Z(z_m, z)$, $h_n\geq h_m$; and $d_Z(z_m, z)\geq d_Z(z_n, z)$, $h_n\geq h_m$.

For the former case, that is, if $d_Z(z_n, z)\geq d_Z(z_m, z)$ and $h_n\geq h_m$, then
$$d_Z(z_n, z_m)\leq 2d_Z(z_n, z)$$
and
\begin{align*}
(x_n|x_m)_o
\geq \log\frac{(d_Z(z_n, z)+h_n)(d_Z(z_m, z)+h_m)}{(2d_Z(z_n, z)+h_n)h}
\geq \log(d_Z(z_m, z)+h_m)-\log(2h).
\end{align*}
This shows that $(x_n|x_m)_o\to \infty$ as $m, n\to\infty$.

For the case $d_Z(z_m, z)\geq d_Z(z_n, z)$ and $h_n\geq h_m$, we have
$$d_Z(z_n, z_m)\leq 2d_Z(z_m, z),$$
and then
\begin{align*}
(x_n|x_m)_o
&\geq \log\frac{(d_Z(z_n, z)+h_n)(d_Z(z_m, z)+h_m)}{h(2d_Z(z_m, z)+h_n)}\\
&\geq \log\frac{(d_Z(z_n, z)+h_n)(d_Z(z_m, z)+h_m)}{d_Z(z_m, z)+h_n}-\log(2h).
\end{align*}
Note that
\begin{align*}
\frac{d_Z(z_m, z)+h_n}{(d_Z(z_n, z)+h_n)(d_Z(z_m, z)+h_m)}
&\leq \frac{d_Z(z_m, z)+h_n}{h_n(d_Z(z_m, z)+h_m)}\\
&\leq \frac{1}{h_n}+\frac{1}{d_Z(z_m, z)+h_m}\to 0
\end{align*}
as  $m, n\to \infty$.
This shows that $(x_n|x_m)_o\to \infty$ as $m, n\to\infty$.

As a result, there is a point $\omega\in \partial_G {\rm Con}_h(Z)$ such that $\{x_n\}$ converges to it.
Next, we show that such point $\omega$ is unique.

Let $\{y_n\}$ be any sequence in ${\rm Con}_h(Z)$ with $y_n=(z_n^\prime, t_n)$ and
$\lim_{n\to\infty}t_n=\infty$.
Without loss of generality, we assume that $\min\{h_n, t_n\}\geq h$ for all $n$. Then
\begin{align*}
(x_n|y_n)_o
= \log\frac{(d_Z(z_n^\prime, z)+h_n)(d_Z(z_n, z)+t_n)}{(d_Z(z_n^\prime, z_n)+h_n\vee t_n)h}.
\end{align*}
Using the same argument with estimating \eqref{eq-23}, we see that
$$(x_n|y_n)_o\to \infty\;\;\mbox{ as}\;\;n\to\infty.$$
This shows that $\{x_n\}$ and $\{y_n\}$ are equivalent, and thus, $\{y_n\}$ converges to $\omega$ as well. This shows that the lemma is true.
\end{proof}
We make a notational convention. In the rest of this section, we use $\omega$ to denote the equivalence class of all sequences $\{x_n=(z_n,h_n)\}$ in ${\rm Con}_h(Z)$ with
$
\lim_{n\to\infty}h_n=\infty.
$

Let $z\in Z$, and let $\gamma=R_{z}|_{[0, \infty)}$. Note that $\omega$ is the endpoint of $\gamma$, and $b_\gamma$ defined by \eqref{def-br} is a Busemann function based at $\omega$.
For any $x=(z^\prime, h)\in {\rm Con}_h(Z)$, a direct calculation gives
\begin{equation}\label{esti-b}
b_{\gamma}(x)= \lim_{t\to \infty}(\rho_h(x, \gamma(t))-t)
= \lim_{t\to \infty}\left(2\log\frac{d_Z(z^\prime, z)+h \vee e^t}{\sqrt{h e^t}}-t\right)
=-\log h.
\end{equation}




\begin{lem}\label{11-12}
Let $z\in Z$. There exists a point $\xi\in \partial_\omega {\rm Con}_h(Z)$ such that for any sequence $\{x_n=(z,h_n)\}$ in $R_z$, if
$$
\lim_{n\to\infty}h_n=0,
$$
then $\{x_n\}$ converges to $\xi$ with respect to $\omega$.
\end{lem}

\begin{proof}
Let $z\in Z$, and let $\gamma=R_z|_{[0, \infty)}$. Assume that $\{x_n=(z,h_n)\}$ is a sequence in $R_z$ such that
$\lim_{n\to\infty}h_n=0$.
It follows from \eqref{esti-b} that
\begin{align*}
(x_n |x_m)_{b_\gamma}=& \frac{1}{2}\big(b_{\gamma}(x_n)+b_{\gamma}(x_m)-\rho_h(x_n, x_m)\big)\\
=& \frac{1}{2}\big(-\log h_m-\log h_m-2\log\frac{h_n\vee h_m}{\sqrt{h_nh_m}}\big)
= -\log(h_n\vee h_m),
\end{align*}
which shows that
$$(x_n | x_m)_{b_\gamma}\to \infty\;\;\mbox{ as}\;\; m, \;n\to \infty.$$
Then $\{x_n\}$ converges to $\xi$ with respect to $\omega$ for some $\xi\in \partial_\omega {\rm Con}_h(Z)$.

Let $\{y_n=(z, t_n)\}$ be a sequence in $R_z$ such that
$\lim_{n\to\infty}t_n=0$.
By using \eqref{esti-b} again, we see that
$$(x_n | y_n)_{b_\gamma}\to \infty\;\;\mbox{ as}\;\; m \to \infty.$$
This shows that $\{x_n\}$ and $\{y_n\}$ are equivalent. This completes the proof of the lemma.
\end{proof}

\subsection{Hyperbolic fillings}

For a metric space $(Z, d_Z)$, we always denote the sets $B(z, r)$ and $\tau B(z, r)$ by
$$
B(z, r)=\{y\in Z : d_Z(y, z)<r\}
$$
and
$$
\tau B(z, r)=\{y\in Z : d_Z(y, z)<\tau r\}
$$
for any $z\in Z$, $r>0$ and $\tau>0$.

Let $(Z, d_Z)$ be a complete metric space. Let $\alpha>10$ and $\tau>3$.  We construct a {\it hyperbolic filling} ${\rm Hyp}(Z)$ of $Z$ as follows.  For each $n\in\mathbb Z$, we select a maximal $\alpha^{-n}$-separated subset $S_n$ of $Z$.  Also, we choose $S_n$ such that $S_m\subset S_n$ whenever $m\geq n$. The existence of such sets is guaranteed by a standard application of Zorn's lemma.
Then for each $n\in\mathbb Z$, the sets $B(z, \alpha^{-n})$, $z\in S_n$, cover $Z$.
The vertex set has the form
$$
V=\bigcup_{n\in\mathbb Z}V_n,
$$
where $V_n=\{(z, n):\; z\in S_n\}.$

To each vertex $v=(z, n)$, we associate the dilated ball $B(v):=\tau B(z, \alpha^{-n})$.
We also define the {\it height
function} $h: V\rightarrow \mathbb Z$ by $h(z, n)=n$.

%

Given two different vertices $v, w\in V$, we say that $w$ is a {\it neighbor} of $v$, denoted by $w\sim v$,
when
$(i)$
$h(v)=h(w)$ and $B(v)\cap B(w)\neq \emptyset$, or
$(ii)$
$h(v)=h(w)+1$ and $B(v)\subset B(w)$.

Define the hyperbolic filling ${\rm Hyp}(Z)$ of $Z$ to be the graph formed by the vertex set $V$ together with the above neighbor relation (edges). Also, we say that $\alpha$ and $\tau$ are the {\it construction parameters} of ${\rm Hyp}(Z)$.

We consider ${\rm Hyp}(Z)$ to be a metric graph, where the edges are unit intervals.
The graph distance between two points $x, y \in {\rm Hyp}(Z)$, denoted by $|x-y|$, is the length of the shortest curve connecting them.

If $(Z, d_Z)$ is  bounded, then the largest integer $k$ with $\diam Z<\alpha^{-k}$ exists, and we denote it by $k_0$. For every $k\leq k_0$, the vertex set $V_k$ consists of one point. We modify the graph ${\rm Hyp}(Z)$ by putting $V_k=\emptyset$ for every $k<k_0$, and call the modified graph the {\it truncated hyperbolic filling} of $Z$, denoted by ${\rm Hyp}_T(Z)$.

For any edge $[v, w]$ in ${\rm Hyp}(Z)$, we extend the height function $h$ to $[v, w]$ by
$$
h(x)=th(w)+(1-t)h(v)
$$
for $x\in [v, w]$ with $|x-v|=t\in[0, 1]$.

The parameters $\alpha>10$ and $\tau>3$ ensure that
$$
\tau\alpha^{-n}\geq 2\tau\alpha^{-n-1}+\alpha^{-n}.
$$
Then for any $v=(z_1, n)\in V_n$ and $w=(z_2, n+1)\in V_{n+1}$, if
$$B(z_1, \alpha^{-n})\cap B(z_2, \alpha^{-n-1})\neq \emptyset,$$
then $B(w)\subset B(v)$ and $w\sim v$.

For any $z\in Z$, there is a geodesic $\gamma_z:\mathbb R\rightarrow {\rm Hyp}(Z)$ such that
$h(\gamma_z(t))=t$ and
\begin{equation}\label{an-z}
z\in B(z_m, \alpha^{-m})
\end{equation}
for any vertex $(z_m, m)\in\gamma_z$. Such $\gamma_z$ may be not unique, in the following, we fix one such geodesic, still denote $\gamma_z$.

Obviously, there is a point $\xi(z)\in \partial_G {\rm Hyp}(Z)$ such that $\xi(z)$ is the endpoint of $\gamma_z|_{[0, \infty)}$. By using the similar argument with  \cite[Lemma 5.11]{BuC}, there exists a point $\hat\omega\in \partial_G {\rm Hyp}(Z)$ such that $\hat\omega$ is the endpoint of $\gamma_z|_{(-\infty, 0]}$ for any $z\in Z$.
Note that $\xi(z)$ and $\hat\omega$ are both independent of the choice of $\gamma_z$ and $\xi(z)\neq \hat\omega$ for any $z\in Z$, see \cite[Section 5]{BuC}.


Define a mapping $\phi: Z\to \partial_{\hat\omega} {\rm Hyp}(Z)$ by letting
$$\phi(z)=\xi(z)$$
for any $z\in Z$.
The following result shows there is an identification of $Z$ and
$\partial_{\hat \omega} {\rm Hyp}(Z)$ when $(Z, d_Z)$ is a complete metric space. See \cite[Theorem 3]{Jordi}, \cite[Proposition 5.13]{BuC} and \cite[Theorem 6.3.1, 6.3.4]{BuSc}.

\begin{thm}\label{Thm-A}
Let $(Z, d_Z)$ be a complete metric space. The hyperbolic filling ${\rm Hyp}(Z)$ is a visual geodesic Gromov $\delta$-hyperbolic space with $\delta=\delta(\alpha, \tau)$, and the mapping
$\phi: Z\to \partial_{\hat\omega} {\rm Hyp}(Z)$
defines an identification of $Z$ with $\partial_{\hat\omega}{\rm Hyp}(Z)$.
Under this identification, $d_Z$ defines a visual metric on $\partial_{\hat\omega} {\rm Hyp}(Z)$ such that $d_Z$ is biLipschitz equivalent to $\alpha^{-(\cdot|\cdot)_b}$ for any Busemann function $b$ based at $\hat\omega$.

Moreover, if $(Z, d_Z)$ is bounded, then the truncated hyperbolic filling ${\rm Hyp}_T(Z)$ is also a visual geodesic Gromov $\delta$-hyperbolic space with $\delta=\delta(\alpha, \tau)$, and there is an identification
$$\partial_G {\rm Hyp}_T(Z)=Z$$
of sets. For any $o\in {\rm Hyp}_T(Z)$, $d_Z$ is biLipschitz equivalent to $\alpha^{-(\cdot|\cdot)_o}$.
\end{thm}

Let $z\in Z$, and let $\gamma:=\gamma_z((-\infty, 0])$ be a geodesic ray with $h(\gamma(0))=0$. It follows from \cite[Lemma 5.12]{BuC} that
\begin{equation}\label{eq-bh-12}
|b(x)-h(x)|\leq 3
\end{equation}
for any $x\in {\rm Hyp}(Z)$,
where $b:=b_\gamma$ is the Busemann function associated to $\gamma$.

For $x, y\in {\rm Hyp}(Z)\cup \partial_{\hat\omega} {\rm Hyp}(Z)$, we define the {\it Gromov product $(x|y)_h$ based at the height function $h$} to be the same as the Gromov product $(x|y)_b$ by replacing the Busemann function $b$ with the height function $h$. More precisely, for any $x, y\in {\rm Hyp}(Z)$,
\beq\label{11-4-1}
(x|y)_h=\frac{1}{2}(h(x)+h(y)-|x-y|).
\eeq

For $\xi\in \partial_{\hat\omega}{\rm Hyp}(Z)$ and $y\in {\rm Hyp}(Z)$, we define
$$
(\xi|y)_h=(y|\xi)_h=\inf\left\{\liminf_{i\rightarrow\infty}(x_i|y)_h:\; \{x_i\}\in\xi\right\},
$$
and for $\zeta, \xi\in\partial_{\hat\omega}{\rm Hyp}(Z)$,
$$
(\zeta|\xi)_h=\inf\left\{\liminf_{i\rightarrow\infty}(x_i|y_i)_h:\; \{x_i\}\in\zeta, \{y_i\}\in\xi\right\}.
$$

For all $x, y\in {\rm Hyp}(Z)\cup \partial_{\hat\omega}{\rm Hyp}(Z)$,
\begin{equation}\label{eq-bh}
|(x|y)_h-(x|y)_b|\leq 3.
\end{equation}
This follows from \eqref{eq-bh-12} and the inequality $$|(x|y)_h-(x|y)_b|\leq (|b(x)-h(x)|+|b(y)-h(y)|)/2.$$

Now let us discuss the relationship between ${\rm Con}_h(Z)$ and ${\rm Hyp}(Z)$.
For this,  we construct a mapping $\sigma$ from ${\rm Con}_h(Z)$ to ${\rm Hyp}(Z)$.
	Define
	$
	\sigma: {\rm Con}_h(Z)\rightarrow {\rm Hyp}(Z)
	$
 as follows: For all $(z, s)\in {\rm Con}_h(Z),$
	\begin{equation}\label{f-con}
		\sigma(z, s)=\gamma_{z}\left(\frac{\log\frac{1}{s}}{\log\alpha}\right).
	\end{equation}
Obviously, $\sigma$ maps each ray $R_{z}$ onto each geodesic $\gamma_{z}$.
For any $(z, s)\in {\rm Con}_h(Z)$,
\begin{equation}\label{sigma-h}
h(\sigma(z, s))=\frac{\log\frac{1}{s}}{\log\alpha}.
\end{equation}
For $s_1, s_2\in(0, \infty)$, if $s_1>s_2$, then $h(\sigma(z, s_1))< h(\sigma(z, s_2))$.

The following result shows that the mapping $\sigma$ defined as above is a rough similarity.

\begin{thm}\label{thm-2.9}
Let $(Z, d_Z)$ be a complete metric space.
The mapping $\sigma: {\rm Con}_h(Z)\rightarrow {\rm Hyp}(Z)$ is a $(1/\log\alpha, C)$-rough similarity with $C=C(\alpha, \tau)$.
\end{thm}

\begin{proof}
Recall that
$$\partial_{\hat \omega} {\rm Hyp}(Z)=Z$$
as sets, and $d_Z$ is a visual metric on $\partial_{\hat \omega} {\rm Hyp}(Z)$ with parameter $\log\alpha$ by Theorem \ref{Thm-A}.

For convenience, we use $z\in Z$ to denote $\xi(z)\in \partial_{\hat \omega}{\rm Hyp}(Z)$. That is, $\xi(z):=z$.
Then there is a $C=C(\alpha,\tau)\geq 1$ such that for any $z, z^\prime\in Z$,
$$
C^{-1}\alpha^{-(z|z^\prime)_b}\leq d_Z(z, z^\prime)\leq C\alpha^{-(z|z^\prime)_b}.
$$
It follows from \eqref{eq-bh} that
\begin{equation}\label{eq-metric-c}
C^{-1}\alpha^{-(z|z^\prime)_h}\leq d_Z(z, z^\prime)\leq C\alpha^{-(z|z^\prime)_h}.
\end{equation}


For any $(z, s)\in {\rm Con}_h(Z)$,  we know that $\sigma(z, s)\in \gamma_{z}$ and $z$ is the endpoint of $\gamma_{z}|_{[0,\infty)}$. So, $\gamma_{z}$ is a geodesic from $\hat\omega$ to $z$.
We claim that
\begin{equation}\label{eq-58}
(z|\sigma(z, s))_h=h(\sigma(z, s))+O_{\delta}(1).
\end{equation}

 Let $\{x_n\}\in z$ be a sequence in $\gamma_{z}$. Then $h(x_n)\to \infty$ as $n\to \infty$.
	Recall that ${\rm Hyp}(Z)$ is Gromov  $\delta$-hyperbolic for some $\delta=\delta(\alpha, \tau)>0$. By Lemma \ref{lem-fun-b} and \eqref{eq-bh}, we see that
	\begin{align}\label{eq-zf}
		(z|\sigma(z, s))_h-6\leq& \liminf_{n\rightarrow+\infty}(x_n|\sigma(z, s))_h\notag\\
		\leq& \limsup_{n\rightarrow+\infty}(x_n|\sigma(z, s))_h\leq (z|\sigma(z, s))_h+600\delta+6.
	\end{align}
Since both $x_n$ and $\sigma(z, s)$ are contained in $\gamma_{z}$,
if $h(x_n)\geq h(\sigma(z, s))$, then
	\begin{equation*}
		(x_n|\sigma(z, s))_h=h(\sigma(z, s)),
	\end{equation*}
which, together with \eqref{eq-zf}, gives
\begin{equation*}
(z|\sigma(z, s))_h=h(\sigma(z, s))+O_{\delta}(1).
\end{equation*}
Then \eqref{eq-58} is true.

For any $(z, s)$ and $(z^\prime, s^\prime)$ in ${\rm Con}_h(Z)$, we infer from \eqref{eq-58} that
$$
(z|\sigma(z, s))_h\geq h(\sigma(z, s))-C(\delta) \ \ \ \text{ and } \ \ \ (z^\prime|\sigma(z^\prime, s^\prime))_h\geq h(\sigma(z^\prime, s^\prime))-C(\delta),
$$
and so, Lemma \ref{lem-eq-b} and \eqref{eq-bh} lead to
		\begin{align*}
		|\sigma(z, s)-\sigma(z^\prime, s^\prime)|=& h(\sigma(z, s))+h(\sigma(z^\prime, s^\prime))\\
&-2\big(h(\sigma(z, s)) \wedge (z|z^\prime)_h \wedge h(\sigma(z^\prime, s^\prime))\big)+O_{\alpha_Z, \tau_Z}(1).
		\end{align*}
It follows from \eqref{sigma-h} that
		\begin{align*}
		|\sigma(z, s)-\sigma(z^\prime, s^\prime)|
		=& -\frac{\log s+\log s^\prime}{\log\alpha }-2\left(\frac{-\log s}{\log\alpha } \wedge (z|z^\prime)_h \wedge \frac{-\log s^\prime}{\log\alpha }\right)+O_{\alpha , \tau }(1)\\
		=& \frac{2}{\log\alpha}\log\frac{\alpha^{-(z|z^\prime)_h}\vee s \vee s^\prime}{\sqrt{ss^\prime}}+O_{\alpha, \tau}(1),
	\end{align*}
and thus, we conclude from \eqref{metric-cone} and \eqref{eq-metric-c} that
	\begin{align*}
	|\sigma(z, s)-\sigma(z^\prime, s^\prime)|=& \frac{2}{\log\alpha}\log\frac{\alpha^{-(z|z^\prime)_h}+s \vee s^\prime}{\sqrt{ss^\prime}}+O_{\alpha, \tau}(1)\\
	=& \frac{2}{\log\alpha}\log\frac{d_Z(z, z^\prime)+s \vee s^\prime}{\sqrt{ss^\prime}}+O_{\alpha, \tau}(1)\\
	=& \frac{1}{\log\alpha} {\rho_h}((z, s), (z^\prime, s^\prime))+O_{\alpha, \tau}(1).
\end{align*}
	
To prove the rough similarity of $\sigma$, it remains to check that the image of ${\rm Con}_h(Z)$ under $\sigma$ is cobounded in ${\rm Hyp}(Z)$. For this, let
	$x\in {\rm Hyp}(Z)$. Then there exists a vertex $v=(z, n)\in V$ such that $|x-v|\leq1$.
Note that $\gamma_{z}(n)\in V_n$ for every $n\in\mathbb{Z}$ and $z\in B(\gamma_{z}(n))$ by \eqref{an-z}.
This ensures that
$$z\in B(v)\cap B(\gamma_{z}(n))$$
 for all $n\in\mathbb{Z}$.
Then $v\sim \gamma_z(m)$ with $m=h(v)$, and so, $\dist(\{v\}, \gamma_{z})\leq 1$.
	We deduce from the fact
	$$\dist(\{x\}, \gamma_{z})\leq |x-v|+\dist(\{v\}, \gamma_{z})$$ that
$$\dist(\{x\}, \gamma_{z})\leq  2.$$
Since $\sigma$ maps $R_{z}\subset {\rm Con}_h(Z)$ onto $\gamma_{z}$ for any $z\in Z$, the image of ${\rm Con}_h(Z)$ under $\sigma$ is $2$-cobounded in ${\rm Hyp}(Z)$.
Therefore, the theorem is proved.
\end{proof}

\section{Proof of Theorem \ref{thm-1}}\label{extension}

In this section, we prove Theorem \ref{thm-1}. Before the proof, let us do some preparation.

For a metric space $(Z,d_Z)$ and a point $x\in Z$, recall that
\begin{equation*}
	R_x=\{x\}\times(0, +\infty).
\end{equation*}

 Assume that $f$ is a $(\theta, \lambda)$-power quasi-symmetry between two metric spaces $(Z, d_Z)$ and $(W, d_W)$, where $\theta\geq 1$ and $\lambda\geq 1$.
 Under this assumption, we are going to define a mapping $f_x$ acting on $R_x$ for each $x\in Z$.

For $x\in Z$, let $S_x\subset \mathbb Z$ denote the set of all $l\in \mathbb Z$ such that the set
	$$
	A_Z(x, l)=\{z\in Z: 2^{-l-1}<d_Z(z, x)\leq 2^{-l}\}
	$$
	is nonempty, where $\mathbb Z$ stands for the integer set. The set $S_x$ can be regarded as the scale spectrum of $Z$ at $x$. Obviously, $S_x$ is non-empty since $\diam Z>0$.
 Similarly, let $S_{x^\prime}$ denote the scale spectrum of $W$ at $x^\prime$, where $x^\prime=f(x)$, and $$A_W(x^\prime, l^\prime)=\{w\in W: 2^{-l^\prime-1}<d_W(w, x^\prime)\leq 2^{-l^\prime}\}.$$
	
	First, we define a function $\phi_x: S_x\to S_{x^\prime}$: For $l\in S_x$, let
	\begin{equation*}
		\phi_x(l)=\sup\big\{l^\prime:\; \mbox{there is}\;\,  y\in Z\;\, \mbox{such that}\;\, 2^{-l-1}<d_Z(y, x) \;\, \mbox{and}\;\,  y^\prime \in A_{W}(x^\prime, l^\prime)\big\},
	\end{equation*}
	where $y^\prime=f(y)$.

	For $l_1, l_2\in S_x$, if $l_1\leq l_2$, obviously, $\phi_x(l_1) \leq \phi_x(l_2)$. This implies that $\phi_x$ is non-decreasing on $S_x$.
	
\begin{lem}\label{eq-5-30}
Let $l\in S_x$. If $y\in A_Z(x, l)$, then there is $l^\prime\in S_{x^{\prime}}$ such that $y^\prime=f(y)\in A_{W}(x^{\prime}, l^{\prime})$ and
$$
|\phi_x(l)-l^\prime|\leq C(\theta, \lambda).
$$
\end{lem}
Let us recall that, here and in the following, $C(\theta, \lambda)$ denotes a constant depending only on the given parameters $\theta$ and  $\lambda$. Although, in different occurrences, their values may be different, we still use the same notation.

\begin{proof}
Let $l\in S_x$ and $y\in A_Z(x, l)$. Then there must be a unique $l^{\prime}\in S_{x^\prime}$ such that $y^{\prime}\in A_{W}(x^{\prime}, l^{\prime})$, where $y^{\prime}=f(y)$.
Also, by the definition of $\phi_x(l)$, there exist a point $y_0\in Z$ and an integer $l_0^\prime\in S_{x^\prime}$ such that
\beq\label{sat-01}
	2^{-l-1}<d_Z(y_0, x),\;\; y_{0}^\prime \in A_{W}(x^\prime, l_0^\prime)  \;\; \text{ and } \;\;  |\phi_x(l)-l_0^\prime|\leq 1.
	\eeq
Subsequently, we have
\beq\label{sat-00}
l^\prime\leq \phi_x(l)\leq l_0^\prime+1.
\eeq

Since $f$ is a $(\theta, \lambda)$-power quasi-symmetry, the fact
$$d_Z(y, x)\leq 2^{-l}<2d_Z(y_0, x)$$ guarantees that
	$$
	\frac{2^{-l^\prime-1}}{2^{-l_0^\prime}}\leq\frac{d_W(y^\prime, x^\prime)}{d_W(y_0^\prime, x^\prime)}\leq 2^\theta\lambda.
	$$
	We infer from \eqref{sat-00} that
$$-1\leq l'_0-l'\leq 1+\theta+\log_2\lambda.$$
Then \eqref{sat-01} leads to
	\begin{equation*}
 |\phi_x(l)-l^\prime| \leq |\phi_x(l)-l_0^\prime|+|l_0^\prime-l^\prime|\leq C(\theta, \lambda).
	\end{equation*}
	This completes the proof.
\end{proof}	
	
	Further, the functions $\phi_x$ possess the following property.
\blem\label{qi}
	For $l_1, l_2\in S_x$, we have
	$$
		\frac{1}{\theta}|l_1-l_2|-C(\theta, \lambda) \leq |\phi_x(l_1)-\phi_x(l_2)| \leq \theta|l_1-l_2|+C(\theta, \lambda).
$$
\elem
	\bpf
 If $l_1=l_2$, the lemma is trivial. If $l_1\not=l_2$,
	without loss of generality, we assume that $l_1< l_2$.
Since $A_Z(x, l_1)\not=\emptyset\not= A_Z(x, l_2)$, let $y_1\in A_Z(x, l_1)$ and $y_2\in A_Z(x, l_2)$. Then we get
	\beq\label{10-7-1}
	2^{l_2-l_1-1}\leq \frac{d_Z(y_1, x)}{d_Z(y_2, x)} \leq 2^{l_2-l_1+1}.
	\eeq
Also, it follows from Lemma \ref{eq-5-30} that there are integers $l_i^\prime$, where $i\in \{1,2\}$, such that $y_i^\prime=f(y_i)\in A_{W}(x^{\prime}, l_i^{\prime})$ and
\beq\label{10-7-2}
|\phi_x(l_i)-l_i^\prime|\leq C(\theta, \lambda).
\eeq
	
By the assumption that $f$ is a $(\theta, \lambda)$-power quasi-symmetry and the estimates in \eqref{10-7-1}, we know that
	$$
2^{\frac{l_2-l_1}{\theta}}C(\theta, \lambda)^{-1} \leq \frac{d_W(y_1^\prime, x^\prime)}{d_W(y_2^\prime, x^\prime)}  \leq2^{{(l_2-l_1)}\theta}C(\theta, \lambda),
	$$
and thus, we infer from the fact $y_i^\prime=f(y_i)\in A_{W}(x^{\prime}, l_i^{\prime})$ for $i\in \{1,2\}$ that
$$
\frac{1}{\theta}|l_1-l_2|-C(\theta, \lambda)\leq |l_1^\prime-l_2^\prime|\leq \theta|l_1-l_2|+C(\theta, \lambda).
$$
Then \eqref{10-7-2} gives
\begin{equation*}
|\phi_x(l_1)-\phi_x(l_2)|\leq |\phi_x(l_1)-l_1^\prime|+|l_1^\prime-l_2^\prime|+ |\phi_x(l_2)-l_2^\prime|
\leq \theta|l_1-l_2|+C(\theta, \lambda)
\end{equation*}
and
\begin{equation*}
|\phi_x(l_1)-\phi_x(l_2)|\geq |l_1^\prime-l_2^\prime|- |\phi_x(l_2)-l_2^\prime|-|\phi_x(l_1)-l_1^\prime|
\geq \frac{1}{\theta}|l_1-l_2|-C(\theta, \lambda),
\end{equation*}
from which the lemma follows.
	\epf
	
	We extend the function $\phi_x: S_x\to S_{x^\prime}$ to a new function on $\mathbb R$ by linear interpolation, which is denoted by $\Phi_x$.  Let
$$M_x :=\sup S_x\;\;\mbox{and}\;\; m_x :=\inf S_x.$$ Then
$$M_x\in \mathbb Z \cup \{+\infty\}\;\;\mbox{and}\;\; m_x\in \mathbb Z \cup \{-\infty\}.$$

Let $x\in Z$. For $t\in \IR$, if $t\in S_x$, let
$$\Phi_x(t)=\phi_x(t).$$

If $S_x$ contains only one element, then $-\infty<m_x=M_x<\infty$. At this time, let $$\Phi_x(t)=\phi_x(M_x)+(t-M_x).$$

If $S_x$ contains at least two elements, then $m_x<M_x$.  For $t\in (m_x,  M_x)\setminus S_x$,  there exists an interval $[l_1, l_2]$ with $l_1, l_2\in S_x$ and $(l_1, l_2)\cap S_x=\emptyset$ such that $t\in[l_1, l_2]$. (For convenience, in the following, we call $[l_1, l_2]$ a {\it nested interval} for $t$. Sometimes, in order to emphasize the based point $x$, we say that $[l_1, l_2]$ is a nested interval for $t$ with respect to $x$. Obviously, for every $t\in (m_x, M_x)\setminus S_x$, its nested interval is unique.)
Now, we define
 \begin{equation}\label{11-25-1}
\Phi_x(t)=(1-\mu_t)\phi_x(l_1)+\mu_t \phi_x(l_2),
\end{equation}
 where $t=(1-\mu_t)l_1+\mu_t l_2$ with $\mu_t\in [0, 1]$. Then equivalently,
\begin{equation}\label{def-phi}
\Phi_x(t)=\frac{l_2\phi_x(l_1)-l_1\phi_x(l_2)}{l_2-l_1}+\frac{\phi_x(l_2)-\phi_x(l_1)}{l_2-l_1}t.
\end{equation}

If $M_x<+\infty$, then $M_x=\max S_x$. For $t\geq M_x$, let $$\Phi_x(t)=\phi_x(M_x)+(t-M_x).$$

If $m_x>-\infty$, then $m_x=\min S_x$. For $t\leq m_x$, let
\begin{equation}\label{eq-30-1}
\Phi_x(t)=\phi_x(m_x)+(t-m_x).
\end{equation}
	
	Obviously, the extended function $\Phi_x$ is non-decreasing and continuous on $\IR$, and satisfies
	\beq\label{mon-01}
	\lim_{t\to -\infty}\Phi_x(t)=-\infty,  \;\;  \lim_{t\to +\infty}\Phi_x(t)=+\infty \;\;  \text{and} \;\;  \Phi_x(\mathbb R)=\mathbb R.
\eeq

\begin{lem}\label{lem-linear}
For an interval $[u, v]\subset\mathbb R$, suppose that it is contained in $(-\infty, m_x]$ with $m_x>-\infty$, or $[l_1, l_2]$ with $l_1,l_2\in S_x$ and
$(l_1, l_2)\cap S_x=\emptyset$, or $[M_x, +\infty)$ with $M_x<+\infty$. For  $t\in [u, v]$, if $t=(1-\mu_t)u+\mu_t v$ with $\mu_t\in[0, 1]$, then
$$
\Phi_x(t)=(1-\mu_t)\Phi_x(u)+\mu_t\Phi_x(v).
$$
\end{lem}

\begin{proof}
Let $t=(1-\mu_t) u+\mu_t v$ with $\mu_t\in[0, 1]$.
If $[u, v]\subseteq (-\infty, m_x]$ with $m_x>-\infty$, then
$$\Phi_x(u)=\phi_x(m_x)+u-m_x,\;\;\Phi_x(v)=\phi_x(m_x)+v-m_x\;\;\mbox{and}\;\;\Phi_x(t)=\phi_x(m_x)+t-m_x.$$
These lead to
$$
\Phi_x(t)=(1-\mu_t)\Phi_x(u)+\mu_t\Phi_x(v).
$$

Similarly, if $[u, v]\subseteq [M_x, +\infty)$ with $M_x<+\infty$, then
$$
\Phi_x(t)=\phi_x(M_x)+(t-M_x)=(1-\mu_t)\Phi_x(u)+\mu_t\Phi_x(v).
$$

If $[u, v]\subseteq [l_1, l_2]$ with $l_1, l_2\in S_x$ and $(l_1, l_2)\cap S_x=\emptyset$, then \eqref{def-phi} ensures that

$$\Phi_x(u)=\frac{l_2\phi_x(l_1)-l_1\phi_x(l_2)}{l_2-l_1}+\frac{\phi_x(l_2)-\phi_x(l_1)}{l_2-l_1}u,$$
$$\Phi_x(v)=\frac{l_2\phi_x(l_1)-l_1\phi_x(l_2)}{l_2-l_1}+\frac{\phi_x(l_2)-\phi_x(l_1)}{l_2-l_1}v $$
and
$$\Phi_x(t)=\frac{l_2\phi_x(l_1)-l_1\phi_x(l_2)}{l_2-l_1}+\frac{\phi_x(l_2)-\phi_x(l_1)}{l_2-l_1}t.$$
These guarantee that
$$\Phi_x(t)=(1-\mu_t) \Phi_x(u)+\mu_t\Phi_x(v),
$$ and hence, the lemma is proved.
\end{proof}

The following result concerning $\Phi_x$ is a generalization of Lemma \ref{qi}.
\begin{lem}\label{lem-phi}
The function $\Phi_x: \mathbb R\to \mathbb R$ is a $(\theta,C)$-rough quasi-isometry, where $C=C(\theta,\lambda)$. More precisely, for any $t_1, t_2\in \mathbb R$, we have
\begin{equation}\label{phi-qi}
		\frac{1}{\theta}|t_1-t_2|-C(\theta,\lambda) \leq |\Phi_x(t_1)-\Phi_x(t_2)| \leq \theta|t_1-t_2|+C(\theta,\lambda).
\end{equation}
\end{lem}

\begin{proof}
Let $t_1$, $t_2\in \mathbb R$. Without loss of generality, we assume that $t_1< t_2$. If $S_x$ contains only one element, or $[t_1, t_2]\subseteq (-\infty, m_x]$ with $m_x>-\infty$, or $[t_1, t_2]\subseteq [M_x,\infty)$ with $M_x<\infty$, then the proofs are obvious since, in these cases, $$|\Phi_x(t_1)-\Phi_x(t_2)|=|t_1-t_2|.$$

In the following, we assume that $S_x$ contains at least two elements, $[t_1, t_2]\nsubseteq (-\infty, m_x]$ with $m_x>-\infty$, and $[t_1, t_2]\nsubseteq [M_x,\infty)$ with $M_x<\infty$.

If $[t_1, t_2]\cap S_x=\emptyset$, then there exist $l_1$, $l_2\in S_x$ such that $[l_1, l_2]$ is the nested interval for both $t_1$ and $t_2$. This implies that for $i\in\{1, 2\}$,
$$
\Phi_x(t_i)=\frac{l_2\phi_x(l_1)-l_1\phi_x(l_2)}{l_2-l_1}+\frac{\phi_x(l_2)-\phi_x(l_1)}{l_2-l_1}t_i,
$$
and so, it follows from Lemma \ref{qi} that
$$ \frac{1}{\theta}|t_1-t_2|-C(\theta,\lambda) \leq \Phi_x(t_2)-\Phi_x(t_1)=\frac{\phi_x(l_2)-\phi_x(l_1)}{l_2-l_1}(t_2-t_1)\leq  \theta|t_1-t_2|+C(\theta,\lambda),$$ where the fact $0<t_2-t_1<l_2-l_1$ is applied.

Now, we assume that $[t_1, t_2]\cap S_x\not=\emptyset$.
Then there are  $l_3$, $l_4\in [t_1, t_2]\cap S_x$ such that
\begin{equation}\label{eq-9-21}
[t_1, t_2]\cap S_x=[l_3, l_4]\cap S_x.
\end{equation} It is possible that $l_3=l_4$. If this case occurs, we regard the closed interval $[l_3,l_4]$ as the singlet $\{l_3(=l_4)\}$.

Since $\phi_x$ is non-decreasing, again, by Lemma \ref{qi}, we see that
\begin{equation*}
\frac{1}{\theta}(l_4-l_3)-C(\theta,\lambda) \leq \phi_x(l_4)-\phi_x(l_3)\leq \theta(l_4-l_3)+C(\theta,\lambda).
\end{equation*}
This shows that \eqref{phi-qi} holds true for the case when $t_1=l_3$ and $t_2=l_4$.

We separate the rest arguments into the following cases.

\bca\label{case-a}
Suppose that $t_1\neq l_3$ and $t_2=l_4$.
\eca
From \eqref{eq-9-21}, we see that $[t_1, l_3)\cap S_x=\emptyset$. If $l_3=m_x>-\infty$, then
\beq\label{10-8-5}
\Phi_x(l_3)-\Phi_x(t_1)=l_3-t_1.
\eeq

If $l_3>m_x$, then there is $k_1\in S_x$ such that $[k_1, l_3]$ is the nested interval for $t_1$. It follows that
\beq\label{10-8-6}
\Phi_x(l_3)-\Phi_x(t_1)=\frac{\phi_x(l_3)-\phi_x(k_1)}{l_3-k_1}(l_3-t_1).
\eeq

Since $t_2=l_4$ and
\begin{equation*}
\Phi_x(t_2)-\Phi_x(t_1)=\phi_x(l_4)-\phi_x(l_3)+\Phi_x(l_3)-\Phi_x(t_1),
\end{equation*}
it follows from Lemma \ref{qi}, together with \eqref{10-8-5} and \eqref{10-8-6}, that
\begin{equation*}
\frac{1}{\theta}(t_2-t_1)-C(\theta, \lambda)\leq |\Phi_x(t_1)-\Phi_x(t_2)|\leq \theta(t_2-t_1)+C(\theta, \lambda).
\end{equation*}

\bca\label{case-b}
Suppose that $t_1=l_3$ and $t_2\neq l_4$.
\eca

Similarly, under this assumption, if  $l_4=M_x<+\infty$, then
\begin{equation*}
\Phi_x(t_2)-\Phi_x(l_4)=t_2-l_4,
\end{equation*}
and if $l_4<M_x$, then there is $k_2\in S_x$ such that $[l_4, k_2]$ is the nested interval for $t_2$, and
\begin{equation*}
\Phi_x(t_2)-\Phi_x(l_4)=\frac{\phi_x(k_2)-\phi_x(l_4)}{k_2-l_4}(t_2-l_4),
\end{equation*}
and so, Lemma \ref{qi} leads to
\begin{equation*}
\frac{1}{\theta}(t_2-t_1)-C(\theta, \lambda)\leq |\Phi_x(t_1)-\Phi_x(t_2)|\leq \theta(t_2-t_1)+C(\theta, \lambda).
\end{equation*}

\bca
Suppose that $t_1\neq l_3$ and $t_2\neq l_4$.
\eca
The similar reasoning as in the discussions in Cases \ref{case-a} and \ref{case-b} shows that
\begin{equation*}\label{emb-condition}
		\Phi_x(l_3)-\Phi_x(t_1)=
		\left\{\begin{array}{cl}
			l_3-t_1,& {\rm if} \;\; l_3=m_x>-\infty \\
			\frac{\phi_x(l_3)-\phi_x(k_3)}{l_3-k_3}(l_3-t_1),& {\rm if} \;\; l_3>m_x
		\end{array}\right.
	\end{equation*}
and
\begin{equation*}\label{emb-condition}
		\Phi_x(t_2)-\Phi_x(l_4)=
		\left\{\begin{array}{cl}
			t_2-l_4,& {\rm if} \;\; l_4=M_x<+\infty\\
			\frac{\phi_x(k_4)-\phi_x(l_4)}{k_4-l_4}(t_2-l_4),& {\rm if} \;\; l_4<M_x,
		\end{array}\right.
	\end{equation*}
where $k_3\in S_x$ and $[k_3, l_3]$ is the nested interval for $t_1$ (resp. $k_4\in S_x$ and $[l_4, k_4]$ is the nested interval for $t_2$).

Since
\begin{equation*}
\Phi_x(t_2)-\Phi_x(t_1)=\Phi_x(t_2)-\Phi_x(l_4)+(\phi_x(l_{4})-\phi_x(l_3))+\Phi_x(l_3)-\Phi_x(t_1),
\end{equation*}
it follows from Lemma \ref{qi} that
$$
\frac{1}{\theta}|t_1-t_2|-C(\theta,\lambda) \leq |\Phi_x(t_1)-\Phi_x(t_2)| \leq \theta|t_1-t_2|+C(\theta,\lambda).
$$
Now, the lemma is proved.
\end{proof}

\blem\label{9-30-1}
Suppose that $x\not= y\in Z$.
Then we have		
\beq\label{log-simi}
\left|\log_2\frac{1}{d_W(x^\prime, y^\prime)}-\Phi_x\left(\log_2\frac{1}{d_Z(x, y)}\right)\right|\leq C(\theta, \lambda)
\eeq
and
\beq\label{phi-xy}
\left|\Phi_x\left(\log_2\frac{1}{d_Z(x, y)}\right)-\Phi_y\left(\log_2\frac{1}{d_Z(x, y)}\right)\right|\leq C(\theta, \lambda).
\eeq
\elem	

\bpf Let $x, y \in Z$ with $x\neq y$, $x^{\prime}=f(x)$ and $y^{\prime}=f(y)$.
	Obviously, there exists $l\in S_x$ such that $y\in A_Z(x, l)$, i.e.,
	\beq\label{10-7-7}	
l \leq \log_2\frac{1}{d_Z(x, y)}<l+1.	
\eeq
Also, we know from Lemma \ref{eq-5-30} that there is $l^\prime\in S_{x^\prime}$ such that $y^\prime\in A_W(x^\prime, l^\prime)$, i.e.,
\begin{equation}\label{20-1}
l^\prime \leq \log_2\frac{1}{d_W(x^\prime, y^\prime)}<l^\prime+1\;\;\mbox{and}\;\;|\Phi_x(l)-l^\prime|\leq C(\theta, \lambda).
\end{equation}
	
Moreover, it follows from \eqref{10-7-7} and Lemma \ref{lem-phi} that
\begin{equation}\label{20-2}
0\leq \Phi_x\left(\log_2\frac{1}{d_Z(x, y)}\right)- \Phi_x(l) \leq \Phi_x(l+1)-\Phi_x(l)\leq C(\theta, \lambda).
\end{equation}

Since
\begin{align*}
\left|\log_2\frac{1}{d_W(x^\prime, y^\prime)}-\Phi_x\left(\log_2\frac{1}{d_Z(x, y)}\right)\right| \leq & \left|\log_2\frac{1}{d_W(x^\prime, y^\prime)}-l'\right|+
\left|l'-\Phi_x(l)\right|\\
&+ \left| \Phi_x(l)- \Phi_x\left(\log_2\frac{1}{d_Z(x, y)}\right) \right|,
\end{align*}
it follows from the inequalities \eqref{20-1} and \eqref{20-2} that
$$
		\left|\log_2\frac{1}{d_W(x^\prime, y^\prime)}-\Phi_x\left(\log_2\frac{1}{d_Z(x, y)}\right)\right|\leq C(\theta, \lambda),
$$
 which is the estimate \eqref{log-simi} in the lemma. This estimate implies that the second estimate, i.e., \eqref{phi-xy}, in the lemma is true as well, and hence, the proof of the lemma is complete.
\epf

\blem\label{mon-00}
Suppose that $x$, $y\in Z$ with $x\not= y$ and $l\in S_x\cup S_y$ with $l< \log_2 (1/d_Z(x, y))$. Then $$|\Phi_x(l)-\Phi_y(l)|\leq C(\theta, \lambda).$$
\elem	
\bpf
Without loss of generality, we assume that $l\in S_x$. We separate the proof into two possibilities: $\log_2 (1/d_Z(x, y))-2 \leq l< \log_2 (1/d_Z(x, y))$
and $l< \log_2 (1/d_Z(x, y))-2$. For the first possibility, since
\begin{align*}
\left|\Phi_x(l)-\Phi_y(l)\right|
\leq& \left|\Phi_x(l)- \Phi_x\left(\log_2 \frac{1}{d_Z(x, y)}\right)\right|\\
&+ \left|\Phi_x\left(\log_2 \frac{1}{d_Z(x, y)}\right)-\Phi_y\left(\log_2  \frac{1}{d_Z(x, y)}\right)\right|\\
&+ \left|\Phi_y\left(\log_2 \frac{1}{d_Z(x, y)}\right)-\Phi_y(l)\right|,
\end{align*}
$$\left|\Phi_x(l)- \Phi_x\left(\log_2 \frac{1}{d_Z(x, y)}\right)\right|\leq \left|\Phi_x(l+2)-\Phi_x(l)\right|
$$
and
$$
 \left|\Phi_y\left(\log_2 \frac{1}{d_Z(x, y)}\right)-\Phi_y(l)\right|\leq \left|\Phi_y(l+2)-\Phi_y(l)\right|,
$$
we infer from \eqref{phi-xy} in Lemma \ref{9-30-1} and Lemma \ref{lem-phi} that
	\beq\label{10-8-3}
	|\Phi_x(l)-\Phi_y(l)|\leq C(\theta, \lambda).
\eeq

For the remaining possibility, that is, $l< \log_2 (1/d_Z(x, y))-2$, we easily know that $$d_Z(x, y)\leq 2^{-l-2}.$$

Since $l\in S_x$,
we see that there exists a point $z\in A_Z(x, l)$ such that
	\begin{equation}\label{eq-zx}
		d_Z(x, y)\leq 2^{-l-2}<2^{-l-1}<d_Z(z, x)\leq 2^{-l},
	\end{equation}
and thus, we obtain from the fact
 $$d_Z(z, x)-d_Z(x, y)\leq d_Z(y, z)\leq d_Z(z, x)+d_Z(x, y)$$ that
	\begin{equation}\label{eq-zy}
	2^{-l-2}< d_Z(y, z)< 2^{-l+1}.
	\end{equation}
	This implies that there is $k\in S_y$ with $l-1\leq k\leq l+1$ such that $z\in A_Z(y, k)$.
	So, by Lemma  \ref{eq-5-30}, there exists $l_y^\prime\in S_{y^\prime}$ such that  $z^\prime\in A_{W}(y^\prime, l_y^\prime)$ and
$$
|\Phi_y(k)-l_y^\prime|\leq C(\theta, \lambda).
$$
Furthermore, by the fact $l-1\leq k\leq l+1$, Lemma \ref{lem-phi} gives
$$
|\Phi_y(k)-\Phi_y(l)|\leq C(\theta, \lambda),
$$
and so,
\begin{equation}\label{eq-10-8}
|\Phi_y(l)-l_y^\prime|\leq C(\theta, \lambda).
\end{equation}

Since $z\in A_Z(x, l)$, again, we deduce from Lemma \ref{eq-5-30} that there exists $l_x^\prime\in S_{x^\prime}$ such that $z^\prime\in A_{W}(x^\prime, l_x^\prime)$ and
\begin{equation}\label{eq-10-8-1}
|\Phi_x(l)-l_x^\prime|\leq C(\theta, \lambda).
\end{equation}

Moreover, the assumption that $f$ is a $(\theta, \lambda)$-power quasi-symmetry, together with \eqref{eq-zx} and \eqref{eq-zy}, guarantees that
\begin{equation*}
\frac{d_W(x^\prime, z^\prime)}{d_W(y^\prime, z^\prime)}\leq 2^{2\theta}\lambda \ \
\text{ and }
\ \
\frac{d_W(y^\prime, z^\prime)}{d_W(x^\prime, z^\prime)}\leq 2^{2\theta}\lambda.
\end{equation*}
Then  we infer from the facts $z^\prime\in A_{W}(x^\prime, l_x^\prime)$ and $z^\prime\in A_{W}(y^\prime, l_y^\prime)$ that
	\begin{equation*}
|l_x^\prime-l_y^\prime|\leq C(\theta, \lambda).
\end{equation*}
Hence we see from \eqref{eq-10-8} and \eqref{eq-10-8-1} that
	\beq\label{10-8-4}
	|\Phi_x(l)-\Phi_y(l)|\leq |\Phi_x(l)-l_x^\prime| +|l_x^\prime-l_y^\prime|+|l_y^\prime-\Phi_y(l)| \leq  C(\theta, \lambda).
\eeq
Now, we conclude from \eqref{10-8-3} and \eqref{10-8-4} that the lemma is true.
\epf

\blem\label{9-30-2}
Suppose that $x,$ $y\in Z$ with $x\not= y$. If $t\geq d_Z(x, y)$, then
	\begin{equation}\label{phi-simi}
		\left|\Phi_x\left(\log_2 \frac{1}{t}\right)-\Phi_y\left(\log_2 \frac1 t\right)\right|\leq C(\theta, \lambda).
	\end{equation}
\elem
\bpf
	It follows from the inequality \eqref{phi-xy} in Lemma \ref{9-30-1} that
	\begin{equation}\label{10-2-2}
		\left|\Phi_x\left(\log_2 \frac{1}{d_Z(x, y)}\right)-\Phi_y\left(\log_2  \frac{1}{d_Z(x, y)}\right)\right|\leq C(\theta, \lambda).
	\end{equation}

Let $t> d_Z(x, y)$. For the proof under this assumption, let $$m_x=\inf S_x,\;\; M_x=\sup S_x,\;\; m_y=\inf S_y\;\;\mbox{and}\;\;M_y=\sup S_y.$$
Since there is an integer $l$ with $l\leq \min\{M_x, M_y\}$ such that $y\in A_Z(x, l)$ and $x\in A_Z(y, l)$, it follows that
\beq\label{10-2-1}
\log_2 \frac{1}{t}<\log_2\frac{1}{d_Z(x, y)}< l+1\leq \min\{M_x, M_y\}+1.
\eeq

First, we consider the possibility: $\min\{M_x, M_y\}\leq \log_2 \frac{1}{t}< \min\{M_x, M_y\}+1$. Then $\min\{M_x, M_y\}<\infty$. Without loss of generality, we assume that $M_x=\min\{M_x, M_y\}$. Since
\begin{align*}
\left|\Phi_x\left(\log_2 \frac{1}{t}\right)-\Phi_y\left(\log_2 \frac1 t\right)\right|
\leq& \left|\Phi_x\left(\log_2 \frac{1}{t}\right)-\Phi_x\left(M_x\right)\right|+\left|\Phi_x\left(M_x\right)-\Phi_y\left(M_x\right)\right|\\
&+\left|\Phi_y\left(M_x\right)-\Phi_y\left(\log_2 \frac1 t\right)\right|
\end{align*}
and
$$M_x<\log_2 \frac{1}{d_Z(x, y)},$$
we know from Lemmas \ref{lem-phi} and \ref{mon-00} that
\beq\label{10-2-3}
\left|\Phi_x\left(\log_2 \frac{1}{t}\right)-\Phi_y\left(\log_2 \frac1 t\right)\right|
\leq C(\theta, \lambda).
\eeq
This shows that \eqref{phi-simi} holds for this possibility.

Next, we consider the other possibility: $\log_2 \frac{1}{t}< \min\{M_x, M_y\}$. In the following, we find the possible positions of $\log_2 \frac{1}{t}$.

If $m_x=-\infty$ (resp. $m_y=-\infty$), obviously, there exists a nested interval for $\log_2 \frac{1}{t}$ with respect to $x$ (resp. with respect to $y$).

If $m_x=M_x$  (resp. $m_y=M_y$), then $\log_2 \frac{1}{t}\in (-\infty, m_x)$ (resp. $\log_2 \frac{1}{t}\in (-\infty, m_y)$). Now, we assume that $m_x<M_x$  (resp. $m_y<M_y$).

If $m_x>-\infty$ (resp. $m_y>-\infty$), then $\log_2 \frac{1}{t}$ is contained in $(-\infty, m_x]$ or $(m_x, M_x)$ (resp. $(-\infty, m_y]$ or $(m_y, M_y)$).
If $\log_2 \frac{1}{t}\in(m_x, M_x)$ (resp. $\log_2 \frac{1}{t}\in(m_y, M_y)$), then there exists a nested interval for $\log_2 \frac{1}{t}$ with respect to $x$ (resp. with respect to $y$).

From the discussions as above, we know that the possible positions of $\log_2 \frac{1}{t}$ are as follows:

\ben
\item[$(i)$]
$\log_2 \frac{1}{t}\in (-\infty, m_x]\cap (-\infty, m_y]$, where $m_x>-\infty$ and $m_y>-\infty$.

\item[$(ii)$]
$\log_2 \frac{1}{t}\in (-\infty, m_x]\cap [k_1, k_2]$, where $m_x>-\infty$, $m_y<M_y$ and $[k_1, k_2]$ is a nested interval for $\log_2 \frac{1}{t}$ with respect to $y$.

\item[$(iii)$]
$\log_2 \frac{1}{t}\in (-\infty, m_y]\cap[l_1, l_2]$, where $m_y>-\infty$, $m_x<M_x$ and $[l_1, l_2]$ is a nested interval for $\log_2 \frac{1}{t}$ with respect to $x$.

\item[$(iv)$]
$\log_2 \frac{1}{t}\in [l_1, l_2]\cap [k_1, k_2]$, where $m_x<M_x$ (resp. $m_y<M_y$) and $[l_1, l_2]$ (resp. $[k_1, k_2]$) is a nested interval for $\log_2 \frac{1}{t}$ with respect to $x$ (resp. with respect to $y$).
\een

We continue the discussions according to the aforementioned possible positions of $\log_2 \frac{1}{t}$.

\begin{case}\label{case-1}
Suppose that $\log_2 \frac{1}{t}\in (-\infty, m_x]\cap (-\infty, m_y]$, where $m_x>-\infty$ and $m_y>-\infty$.
\end{case}

Without loss of generality, we assume that $m_x\geq m_y$.  Since \eqref{eq-30-1} gives
$$\Phi_x\left(\log_2 \frac{1}{t}\right)-\Phi_y\left(\log_2 \frac{1}{t}\right)=\Phi_x(m_x)-\Phi_y(m_y)-m_x+m_y
$$
and
$$\Phi_x(m_y)=\Phi_x(m_x)+\left(m_y-m_x\right),$$
and since the fact $x\in A_Z(y, l)$ leads to $m_y\leq l\leq \log_2 (1/d_Z(x, y))$,
we know from Lemma \ref{mon-00} that
\beq\label{10-2-4}
\left|\Phi_x\left(\log_2 \frac{1}{t}\right)-\Phi_y\left(\log_2 \frac{1}{t}\right)\right|\leq C(\theta, \lambda).
\eeq

\begin{case}\label{case-2}
Suppose that $\log_2 \frac{1}{t}\in (-\infty, m_x]\cap [k_1, k_2]$, where $m_x>-\infty$, $m_y<M_y$ and $[k_1, k_2]$ is a nested interval for $\log_2 \frac{1}{t}$ with respect to $y$.
\end{case}

 Since $\log_2 \frac{1}{t}\in (-\infty, m_x]\cap [k_1, k_2]$, we know that $m_x\geq k_2$ or $k_1\leq m_x< k_2$.

If $m_x\geq k_2$, then $[k_1, k_2]\subset (-\infty, m_x]$. Since there is  $\mu_1\in[0, 1]$ such that
$$\log_2 \frac{1}{t}=(1-\mu_1)k_1+\mu_1 k_2,$$
by Lemma \ref{lem-linear}, we have
$$
\Phi_x\left(\log_2 \frac{1}{t}\right)=(1-\mu_1)\Phi_x(k_1)+\mu_1\Phi_x(k_2)
$$
and
$$
\Phi_y\left(\log_2 \frac{1}{t}\right)=(1-\mu_1)\Phi_y(k_1)+\mu_1\Phi_y(k_2).
$$

We divide the arguments into two cases: $k_2< \log_2 (1/d_Z(x, y))$ and $k_2\geq \log_2 (1/d_Z(x, y))$. For the first case, we infer from Lemma \ref{mon-00} that
\begin{align}
\left|\Phi_x\left(\log_2 \frac{1}{t}\right)-\Phi_y\left(\log_2 \frac{1}{t}\right)\right| & \leq  (1-\mu_1)|\Phi_x(k_1)-\Phi_y(k_1)|+\mu_1|\Phi_x(k_2)-\Phi_y(k_2)| \notag\\
& \leq  C(\theta, \lambda).
\end{align}

For the remaining case, that is, $k_2\geq \log_2 (1/d_Z(x, y))$, similarly, we know that there is $\mu_2\in[0, 1]$ such that
\begin{align*}
\left|\Phi_x\left(\log_2 \frac{1}{t}\right)-\Phi_y\left(\log_2 \frac{1}{t}\right)\right|
\leq& (1-\mu_2)|\Phi_x(k_1)-\Phi_y(k_1)|\\
 &+\mu_2\Big|\Phi_x\left(\log_2 \frac{1}{d_Z(x, y)}\right)-\Phi_y\left(\log_2 \frac{1}{d_Z(x, y)}\right)\Big|.
\end{align*}
Again, we deduce from Lemma \ref{mon-00}, together with \eqref{10-2-2}, that
\beq\label{10-2-9}
\left|\Phi_x\left(\log_2 \frac{1}{t}\right)-\Phi_y\left(\log_2 \frac{1}{t}\right)\right| \leq  C(\theta, \lambda).
\eeq

If $k_1 \leq m_x< k_2$, then $\log_2 \frac{1}{t}\in [k_1, m_x]$, and so, Lemma \ref{lem-linear} guarantees that there is $\mu_3\in [0,1]$ such that
$$
\Phi_x\left(\log_2 \frac{1}{t}\right)=(1-\mu_3)\Phi_x(k_1)+\mu_3\Phi_x(m_x)
$$
and
$$
\Phi_y\left(\log_2 \frac{1}{t}\right)=(1-\mu_3)\Phi_y(k_1)+\mu_3\Phi_y(m_x).
$$
Since \eqref{10-2-1} implies that $m_x < \log_2 1/d_Z(x, y)$, once more, Lemma \ref{mon-00} ensures that
\begin{align}\label{10-2-5}
\left|\Phi_x\left(\log_2 \frac{1}{t}\right)-\Phi_y\left(\log_2 \frac{1}{t}\right)\right|
& \leq (1-\mu_3)|\Phi_x(k_1)-\Phi_y(k_1)| +\mu_3|\Phi_x(m_x)-\Phi_y(m_x)|\notag\\
& \leq C(\theta, \lambda).
\end{align}

\begin{case}\label{case-3}
Suppose that $\log_2 \frac{1}{t}\in (-\infty, m_y]\cap[l_1, l_2]$, where $m_y>-\infty$, $m_x<M_x$ and $[l_1, l_2]$ is a nested interval for $\log_2 \frac{1}{t}$ with respect to $x$.
\end{case}

Similar arguments as in Case \ref{case-2} show that
\beq\label{10-2-6}
\left|\Phi_x\left(\log_2 \frac{1}{t}\right)-\Phi_y\left(\log_2 \frac{1}{t}\right)\right| \leq  C(\theta, \lambda).
\eeq

\begin{case}\label{case-4}
Suppose that $\log_2 \frac{1}{t}\in [l_1, l_2]\cap [k_1, k_2]$, where $m_x<M_x$ $($resp. $m_y<M_y$$)$ and $[l_1, l_2]$ $($resp. $[k_1, k_2]$$)$ is a nested interval for $\log_2 \frac{1}{t}$ with respect to $x$ $($resp. with respect to $y$$)$.
\end{case}

If $\log_2 \frac{1}{t}\in [l_1, l_2]\cap [k_1, k_2]$, we only need to consider the cases: $l_1\leq k_1\leq k_2\leq l_2$ and $l_1\leq k_1\leq l_2\leq k_2,$
since the discussions for the other two cases, those are $k_1\leq l_1\leq l_2\leq k_2$ and $k_1\leq l_1\leq k_2\leq l_2$, are very similar. It follows from \eqref{10-2-1} that $l_1\vee k_1<\log_21/d_Z(x,y)$. Also, according to the  discussions in Case \ref{case-2}, we may assume that $k_2\vee l_2<  \log_2 (1/d_Z(x, y))$.

For the case $l_1\leq k_1\leq k_2\leq l_2$, that is, $[k_1, k_2]\subset [l_1, l_2]$, it follows from Lemma \ref{lem-linear} that there is $\mu_4\in [0,1]$ such that
\begin{align*}
\left|\Phi_x\left(\log_2 \frac{1}{t}\right)-\Phi_y\left(\log_2 \frac{1}{t}\right)\right|
\leq (1-\mu_4)|\Phi_x(k_1)-\Phi_y(k_1)|+\mu_4|\Phi_x(k_2)-\Phi_y(k_2)|,
\end{align*}
and then,
Lemma \ref{mon-00} gives
\beq\label{10-3-1}
\left|\Phi_x\left(\log_2 \frac{1}{t}\right)-\Phi_y\left(\log_2 \frac{1}{t}\right)\right|
\leq C(\theta, \lambda).
\eeq

For the remaining case, that is, $l_1\leq k_1\leq l_2\leq k_2$, we know from Lemma \ref{lem-linear} that there is $\mu_5\in [0,1]$ such that
\beqq
\left|\Phi_x\left(\log_2 \frac{1}{t}\right)-\Phi_y\left(\log_2 \frac{1}{t}\right)\right|
\leq  (1-\mu_5)|\Phi_x(k_1)-\Phi_y(k_1)|+\mu_5|\Phi_x(l_2)-\Phi_y(l_2)|,\\ \nonumber
\eeqq
and thus, Lemma \ref{mon-00} gives
\beq\label{10-2-7}
\left|\Phi_x\left(\log_2 \frac{1}{t}\right)-\Phi_y\left(\log_2 \frac{1}{t}\right)\right|
\leq  C(\theta, \lambda).
\eeq

Now, we conclude from \eqref{10-2-2}, \eqref{10-2-3}$-$\eqref{10-2-7} that \eqref{phi-simi} holds true for all $t\geq d_Z(x, y)$, and hence, the lemma is proved.
\epf

Now, we are ready to define the mappings $f_x$ based on $\Phi_x$. For $(x, t)\in \text{Con}_h(Z)$,  let
	\begin{equation}\label{eq-def-fx}
		f_x(x, t)=\left(x^\prime, 2^{-\Phi_x\left(\log_2\frac{1}{t}\right)}\right).
	\end{equation} Recall here that $x^{\prime}=f(x)$.
Then it follows from  \eqref{11-25-1}$-$\eqref{mon-01} that $f_x(R_x)=R_{x^\prime}$ for every $x\in Z$.

The following lemma concerning the mappings $f_x$ is a generalization of \cite[Lemma 7.3]{BSC} since neither $(Z,d_Z)$ nor $(W,d_W)$ is assumed to be bounded.

\begin{lem}\label{lem-fx}
Suppose that $f$ is a $(\theta, \lambda)$-power quasi-symmetry between two metric spaces $(Z, d_Z)$ and $(W, d_W)$, where $\theta\geq 1$ and $\lambda\geq 1$.
For $x\in Z$, the mapping $f_x: R_x\to R_{x^{\prime}}$ defined as in \eqref{eq-def-fx} satisfies the following properties:
	
	$(1)$ There exists a constant $k=k(\theta, \lambda)\geq 0$ such that for every $x\in Z$, the mapping $f_x$ is a $(\theta, k)$-rough quasi-isometry.
	
	$(2)$ For $x, y\in Z$ with $x\neq y$, if $t=d_Z(x, y)$ and $f_x(x, t)=(x^\prime, t^\prime)$, then
	$$
	C(\theta, \lambda)^{-1}d_W(x^\prime, y^\prime)\leq  t^\prime \leq C(\theta, \lambda)d_W(x^\prime, y^\prime).
	$$
	
	$(3)$ For $t_1, t_2\in (0, \infty)$, let $f_x(x, t_1)=(x^\prime, {t_1}^\prime)$ and $f_x(x, t_2)=(x^\prime, {t_2}^\prime)$.
 If
	$
	t_1\leq t_2,
	$
	then
	$$
	{t_1}^\prime\leq {t_2}^\prime.
	$$
	
	$(4)$ For $x, y\in Z$ with $x\neq y$, if $t\geq d_Z(x, y)$, then
	$$
	{\rho_h}(f_x(x, t), f_y(y, t))=O_{\theta, \lambda}(1).
	$$
\end{lem}

\begin{proof}
For $(1)$, let $t_1, t_2\in (0, \infty)$ with $t_1\leq t_2$. Then for $x\in Z$,
$$\Phi_x\left(\log_2\frac{1}{t_1}\right)\geq \Phi_x\left(\log_2\frac{1}{t_2}\right).
$$
Elementary calculations show that
	\begin{align*}
		{\rho_h}(f_x(x, h_1),f_x(x, h_2))
		=&2\log\frac{2^{-\Phi_x\left(\log_2\frac{1}{t_1}\right)}\vee 2^{-\Phi_x\left(\log_2\frac{1}{t_2}\right)}}{\sqrt{2^{-\Phi_x\left(\log_2\frac{1}{t_1}\right)-\Phi_x\left(\log_2\frac{1}{t_2}\right)}}}\\
		=&\left(\Phi_x\left(\log_2\frac{1}{t_1}\right)-\Phi_x\left(\log_2\frac{1}{t_2}\right)\right)\log 2.
	\end{align*}
	By Lemma \ref{lem-phi}, we have
	$$
	{\rho_h}(f_x(x, t_1), f_x(x, t_2))\leq(\log_2 t_2-\log_2 t_1)\theta\log2+ C(\theta, \lambda)=\theta{\rho_h}((x, t_1), (x, t_2))+ C(\theta, \lambda)
	$$
	and
	$$
	{\rho_h}(f_x(x, t_1), f_x(x, t_2))\geq(\log_2 t_2-\log_2 t_1)\frac{\log2}{\theta}-C(\theta, \lambda)=\frac{1}{\theta}{\rho_h}((x, t_1), (x, t_2))-C(\theta, \lambda).
	$$
	Thus each mapping $f_x: R_x\to R_{x^\prime}$ is a $(\theta, k)$-rough quasi-isometry with $k=C(\theta, \lambda)$.\medskip
	
	For $(2)$, let $x, y \in Z$ with $x\neq y$, $t=d_Z(x, y)$, and let $f_x(x, t)=(x^\prime, t^\prime)$. Then the inequality \eqref{log-simi} in Lemma \ref{9-30-1} gives
	$$
	C(\theta, \lambda)^{-1}d_W(x^\prime, y^\prime)\leq t^\prime=2^{-\Phi_x\left(\log_2\frac{1}{t}\right)} \leq C(\theta, \lambda)d_W(x^\prime, y^\prime),
	$$ which is what we need.

	For $(3)$, it is clear since $\Phi_x$ is non-decreasing.
	
	For $(4)$, let $x, y \in Z$ with $x\neq y$, and let $t\geq d_Z(x, y)$.
it follows from the inequality \eqref{log-simi} in Lemma \ref{9-30-1} and the fact that $\Phi_x$ is non-decreasing that
	$$\log_2\frac{1}{d_W(x^\prime, y^\prime)} \geq \Phi_x\left(\log_2\frac{1}{d_Z(x, y)}\right)-C(\theta, \lambda)\geq \Phi_x\left(\log_2\frac{1}{t}\right)-C(\theta, \lambda).$$
	Then we have
$$
d_W(x^\prime, y^\prime)\leq C(\theta, \lambda) \cdot 2^{-\Phi_x\left(\log_2\frac{1}{t}\right)},
$$
which gives
$$
\frac{d_W(x^\prime, y^\prime)+2^{-\Phi_x\left(\log_2\frac{1}{t}\right)}\vee 2^{-\Phi_y\left(\log_2\frac{1}{t}\right)}}{\sqrt{2^{-\Phi_x\left(\log_2\frac{1}{t}\right)-\Phi_y\left(\log_2\frac{1}{t}\right)}}}
\leq  C(\theta, \lambda)2^{\frac{\left|\Phi_x\left(\log_2\frac{1}{t}\right)-\Phi_y\left(\log_2\frac{1}{t}\right)\right|}{2}}.
$$		
Thus we know from
$${\rho_h}(f_x(x, t), f_y(y, t))
		= 2\log\frac{d_W(x^\prime, y^\prime)+2^{-\Phi_x\left(\log_2\frac{1}{t}\right)}\vee 2^{-\Phi_y\left(\log_2\frac{1}{t}\right)}}{\sqrt{2^{-\Phi_x\left(\log_2\frac{1}{t}\right)-\Phi_y\left(\log_2\frac{1}{t}\right)}}}$$
that
\beqq
{\rho_h}(f_x(x, t), f_y(y, t)) \leq C(\theta, \lambda)+\log2\cdot \left|\Phi_x\left(\log_2\frac{1}{t}\right)-\Phi_y\left(\log_2\frac{1}{t}\right)\right|.
\eeqq
	Therefore, Lemma \ref{9-30-2} leads to
	$$		{\rho_h}(f_x(x, t), f_y(y, t))=O_{\theta, \lambda}(1),$$	
which shows that the statement $(4)$ of the lemma is true, and hence, the lemma is proved.
\end{proof}

Suppose that $f$ is a power quasi-symmetry (which includes snowflake mappings) between two metric spaces $(Z, d_Z)$ and $(W, d_W)$. Then
we define
$
\widehat{f}: \text{Con}_h(Z)\to \text{Con}_h(W)
$
as follows: For $(x, t)\in \text{Con}_h(Z)$, let
\begin{equation*}
	\widehat{f}(x, t)=f_x(x, t),
\end{equation*}
where the mapping $f_x$ is defined as in \eqref{eq-def-fx}.

Now, we are ready to prove Theorem \ref{thm-1}.

\begin{proof}[{\bf Proof of Theorem \ref{thm-1}}]
	Assume that $f: Z\to W$ is a $(\theta, \lambda)$-power quasi-symmetry with $\theta\geq 1$ and $\lambda\geq 1$.
	First, we know from \eqref{mon-01} that for every $x\in Z$, $f_x(R_x)=R_{x^\prime}$, where $x^\prime=f(x)$.
	This leads to
	$$
	\widehat{f}(\text{Con}_h(Z))=\bigcup_{x\in Z}f_x(R_x)=\bigcup_{x\in Z}R_{x^\prime}=\text{Con}_h(W),
	$$
	which implies that $\widehat f(\text{Con}_h(Z))$ is cobounded in $\text{Con}_h(W)$.
	
	Next, we show that for any $q_1, q_2\in \text{Con}_h(Z)$,
	\begin{equation}\label{eq-6-5-1}
		\theta^{-1} {\rho_h}(q_1, q_2)-k\leq{\rho_h}\left(\widehat f(q_1), \widehat f(q_2)\right)\leq \theta{\rho_h}(q_1, q_2)+k,
	\end{equation}
	where $k=k(\theta, \lambda)$.
	
For the proof, let $q_1=(x_1, t_1)$, $q_2=(x_2, t_2)$, and let $q_1^\prime=\widehat{f}(q_1)=(x_1^\prime, t_1^\prime)$, $q_2^\prime=\widehat{f}(q_2)=(x_2^\prime, t_2^\prime)$. Set $t=d_Z(x_1, x_2)\vee t_1 \vee t_2$, $p_1=(x_1, t)$, $p_2=(x_2, t)$, $p_1^\prime=\widehat{f}(p_1)=(x_1^\prime, t^\prime)$, and $p_2^\prime=\widehat{f}(p_2)=(x_2^\prime, \widetilde{t}^\prime)$. Then elementary calculations ensure that
$${\rho_h}(p_1, p_2)\leq \log4,$$ and hence,
	$$
	{\rho_h}(q_1, q_2)\leq {\rho_h}(q_1, p_1)+ {\rho_h}(q_2, p_2)+\log 4.
	$$

	Since $t\leq d_Z(x_1, x_2)+t_1\vee t_2$, we have
	$${\rho_h}(q_1, p_1)+ {\rho_h}(q_2, p_2)=2\log\frac{t}{\sqrt{t_1t_2}}\leq {\rho_h}(q_1, q_2).$$
These show that
	\begin{equation}\label{eq-6-5}
		{\rho_h}(q_1, q_2)= {\rho_h}(q_1, p_1)+ {\rho_h}(q_2, p_2)+O(1).
	\end{equation}

Next, we prove the following estimate
\beq\label{wedn-1}
{\rho_h}(q_1^\prime, q_2^\prime)={\rho_h}(q_1^\prime, p_1^\prime)+{\rho_h}(q_2^\prime, p_2^\prime)+O_{\theta, \lambda}(1).
\eeq

For the proof, we consider two possibilities: $t=d_Z(x_1, x_2)$ and $t=t_1 \vee t_2$ since $t=d_Z(x_1, x_2)\vee t_1 \vee t_2$. For the first possibility, it follows from Lemma \ref{lem-fx}(2) that
\begin{equation}\label{eq-28-2}
	C(\theta, \lambda)^{-1}d_W(x_1^\prime, x_2^\prime)\leq \min \{t^\prime, \widetilde{t}^\prime\}\leq \max\{t^\prime, \widetilde{t}^\prime\} \leq C(\theta, \lambda)d_W(x_1^\prime, x_2^\prime).
\end{equation}
Also, by Lemma \ref{9-30-2}, we know that
\beqq
	\left|\Phi_{x_1}\left(\log_2\frac{1}{t}\right)-\Phi_{x_2}\left(\log_2\frac{1}{t}\right)\right|\leq C(\theta, \lambda),
\eeqq
and then, we get
\begin{align}\label{9-29-1}
C(\theta, \lambda)^{-1}\widetilde{t}^\prime=C(\theta, \lambda)^{-1}2^{-\Phi_{x_2}\left(\log_2\frac{1}{t}\right)} \leq & t^\prime=2^{-\Phi_{x_1}\left(\log_2\frac{1}{t}\right)}\notag\\
\leq & C(\theta, \lambda)2^{-\Phi_{x_2}\left(\log_2\frac{1}{t}\right)}=C(\theta, \lambda)\widetilde{t}^\prime.
\end{align}
Since $t\geq t_1\vee t_2$, by Lemma \ref{lem-fx}(3), we have that $t^\prime\geq t_1^\prime$ and $\widetilde{t}^\prime\geq t_2^\prime$.
Then \eqref{9-29-1} leads to
\begin{equation}\label{eq-28-3}
t_1^\prime\vee t_2^\prime\leq C(\theta, \lambda)\min\{t^\prime, \widetilde{t}^\prime\}\leq C(\theta, \lambda)\sqrt{t^\prime \widetilde{t}^\prime},
\end{equation}
which, together with \eqref{eq-28-2}, implies that
	\begin{align*}
		{\rho_h}(q_1^\prime, q_2^\prime)
		=&2\log\frac{d_W(x_1^\prime, x_2^\prime)+t_1^\prime \vee t_2^\prime}{\sqrt{t_1^\prime t_2^\prime}}
		= 2\log\frac{\sqrt{t^\prime \widetilde{t}^\prime}}{\sqrt{t_1^\prime t_2^\prime}}+O_{\theta, \lambda}(1) \\
		=& {\rho_h}(q_1^\prime, p_1^\prime)+{\rho_h}(q_2^\prime, p_2^\prime)+O_{\theta, \lambda}(1),
	\end{align*} which is what we need.
	
For the remaining possibility, that is, $t=t_1 \vee t_2$, without loss of generality, we assume that $t=t_1$. Then $t\geq d_Z(x_1, x_2)$ and $q_1^\prime=p_1^\prime$, and thus, Lemma \ref{lem-fx}(4) gives
	$$
	{\rho_h}(p_1^\prime, p_2^\prime)={\rho_h}(f_{x_1}(x_1, t), f_{x_2}(x_2, t))= O_{\theta, \lambda}(1).
	$$
	By the triangle inequality, we obtain that
	\begin{align*}
		{\rho_h}(q_2^\prime, p_2^\prime)-C(\theta, \lambda) \leq {\rho_h}(q_1^\prime, q_2^\prime)={\rho_h}(p_1^\prime, q_2^\prime)\leq {\rho_h}(q_2^\prime, p_2^\prime)+C(\theta, \lambda),
	\end{align*}
	that is, $${\rho_h}(q_1^\prime, q_2^\prime)={\rho_h}(q_2^\prime, p_2^\prime)+O_{\theta, \lambda}(1).$$ This proves \eqref{wedn-1} since in this case $q_1^\prime=p_1^\prime$.
	
	On the one hand, we have that
	\begin{align*}
		{\rho_h}(q_1^\prime, q_2^\prime)&= {\rho_h}(q_1^\prime, p_1^\prime)+{\rho_h}(q_2^\prime, p_2^\prime)+O_{\theta, \lambda}(1)\;\;\;\;\; \;\;\;\;\mbox{(by \eqref{wedn-1})}
\\
		&\leq \theta({\rho_h}(q_1, p_1)+{\rho_h}(q_2, p_2))+C(\theta, \lambda)\;\;\;\;\mbox{(by Lemma \ref{lem-fx}(1))}
\\
		&\leq \theta{\rho_h}(q_1, q_2)+C(\theta, \lambda).\;\;\;\;\;\;\;\;\; \;\;\;\;\;\;\;\;\;\;\;\;\;\;\;\;\mbox{(by \eqref{eq-6-5})}
	\end{align*}

	On the other hand, we obtain that
	\begin{align*}
		{\rho_h}(q_1^\prime, q_2^\prime)\geq \frac{1}{\theta}({\rho_h}(q_1, p_1)+{\rho_h}(q_2, p_2))-C(\theta, \lambda)
		\geq \frac{1}{\theta}{\rho_h}(q_1, q_2)-C(\theta, \lambda).
	\end{align*}
	These prove \eqref{eq-6-5-1} with $k=C(\theta, \lambda)$, and hence, $\widehat{f}$ is a $(\theta, k)$-rough quasi-isometry. This shows that the first part of Theorem \ref{thm-1} is true.

Similarly, we can prove that the second part of Theorem \ref{thm-1} holds true as well. Based on the arguments in Lemmas \ref{qi} and \ref{lem-phi}, if $f: Z\to W$ is an $(\alpha, C)$-snowflake mapping, then for each $x\in Z$, $\Phi_x$ is an $(\alpha, k_0)$-rough similarity with $k_0=k_0(\alpha, C)$. Hence, by Lemma \ref{lem-fx}, $f_x$ is an $(\alpha, k^\prime)$-rough similarity with $k^\prime=k^\prime(\alpha, C)$ for each $x\in Z$. Using the proof of the first part again, $\widehat{f}$ is an $(\alpha, k^\prime)$-rough similarity with $k^\prime=k^\prime(\alpha, C)$. The proof of Theorem \ref{thm-1} is complete.
\end{proof}

\section{Proof of Theorem \ref{thm-2}}\label{sec-5}

Assume that $(Z, d_Z)$ is a complete metric space. Let $\omega\in \partial_G {\rm Con}_h(Z)$ be given in Lemma \ref{union}.
Define a mapping $\psi: Z\to \partial_\omega {\rm Con}_h(Z)$ by letting
\beq\label{11-12-2}
\psi(z)=\xi,
\eeq
where $\xi\in \partial_\omega {\rm Con}_h(Z)$ is the equivalence class such that each element $\{(z, h_n)\}$ in the class satisfies $\lim_{n\to \infty}h_n=0$ and converges to $\xi$ with respect to $\omega$, see Lemma \ref{11-12}.

Fix a point $z_0\in Z$, let $\gamma=R_{z_0}|_{[0,\infty)}$. Then $\omega$ is the endpoint of $\gamma$.
In the following, we fix the Busemann function $b:=b_\gamma\in\mathcal B(\omega)$.

\begin{proof}[{\bf Proof of Theorem \ref{thm-2}}]
To prove that $\partial_\omega {\rm Con}_h(Z)$ and $Z$ can be identified as sets, it suffices to show that $\psi: Z\to \partial_\omega {\rm Con}_h(Z)$ is a bijection.

First, we prove that $\psi$ is injective.
For this, let $z_1$, $z_2\in Z$, and let $\{x_n=(z_1,h_n)\}$ and $\{y_n=(z_2, s_n)\}$ be two sequences in ${\rm Con}_h(Z)$ with
$$
\lim_{n\to \infty}h_n=\lim_{n\to \infty}s_n=0.
$$
Then $\{x_n\}$ converges to $\psi(z_1)$ with respect to $\omega$, and $\{y_n\}$ converges to $\psi(z_2)$ with respect to $\omega$.
It follows from \eqref{esti-b} that
\begin{align}\label{imp}
(x_n| y_n)_b=-\log(d_Z(z_1, z_2)+h_n\vee s_n),
\end{align}
and then, by letting $n\to \infty$ and using Lemma \ref{lem-fun-b}, we know that there is a constant $C=C(\delta)\geq 1$ such that
\begin{equation}\label{visual-metric}
C^{-1} d_Z(z_1, z_2)\leq e^{-(\psi(z_1) | \psi(z_2))_b}\leq C d_Z(z_1, z_2),
\end{equation}
which shows that $\psi$ is injective.

Second, we show that $\psi$ is surjective. To reach this goal, let $\xi\in \partial_\omega {\rm Con}_h(Z)$, and let $\{x_n=(z_n, h_n)\}$ be a sequence converging to $\xi$ with respect to $\omega$. Since
\begin{equation*}
(x_n| x_m)_b=-\log(d_Z(z_n, z_m)+h_n\vee h_m),
\end{equation*}
by letting $m$, $n\to \infty$, we have
\begin{equation}\label{to0}
d_Z(z_n, z_m)\to 0 \ \ \text{ and } \ \ h_n, h_m\to 0.
\end{equation}
Then $\{z_n\}$ is a Cauchy sequence in $Z$, and so, it converges to a point $z\in Z$ by the assumption that $(Z, d_Z)$ is a complete metric space.

We claim that $\psi(z)=\xi$. For the proof, let $\{x'_n=(z, u_n)\}$ be a sequence in ${\rm Con}_h(Z)$
with $\lim_{n\to \infty}u_n=0$. Then $\{x'_n\}$ converges to $\psi(z)$ with respect to $\omega$.  We infer from \eqref{to0} that
\begin{equation*}
(x'_n| x_n)_b=-\log(d_Z(z,z_n)+u_n\vee h_n)\to \infty \ \ \text{ as }  \ \ n\to\infty,
\end{equation*}
which shows that $\{x'_n\}$ and $\{x_n\}$ are equivalent  with respect to $\omega$. Hence, $\psi(z)=\xi$ and $\psi$ is surjective.
Now we conclude that $\psi: Z\to \partial_\omega {\rm Con}_h(Z)$ is a bijection.

For $z_1$, $z_2\in Z$, let
$$
d(\psi(z_1), \psi(z_2)):=d_Z(z_1, z_2).
$$
Then it follows from \eqref{visual-metric} that $d_Z$ induces a visual metric $d$ on $\partial_\omega {\rm Con}_h(Z)$ based at $\omega$ with parameter $1$. The proof of this theorem is complete.
\end{proof}

\begin{rem}\label{rem-3-1}
Suppose that $(Z, d_Z)$ is a complete bounded metric space. Let $o=(z, h)\in {\rm Con}(Z)$. For any $x=(z_1, h_1)$, $y=(z_2, h_2)\in {\rm Con}_h(Z)$,
\begin{equation}\label{rem-eq-1}
(x|y)_o=-\log(d_Z(z_1, z_2)+h_1\vee h_2)+\log\frac{(d_Z(z_1, z)+h_1\vee h)(d_Z(z_2, z)+h_2\vee h)}{h}.
\end{equation}

If a sequence $\{(z_n, h_n)\}$ does not converge to $\omega$, then $h_n$ does not go to infinity as $n\to\infty$ by Lemma \ref{union}. So, any sequence $\{(z_n, h_n)\}$ converges to a point in $\partial_G{\rm Con}_h(Z)\setminus \{\omega\}$ if and only if
\begin{equation*}
d_Z(z_n, z_m)+h_n\vee h_m\to 0 \ \ \text{ as } m, n\to\infty.
\end{equation*}
This implies that $\partial_G{\rm Con}_h(Z)\setminus \{\omega\}=\partial_G {\rm Con}(Z)$, and then, by Proposition \ref{prop-Gromov}, as sets,
$$
\partial_G {\rm Con}(Z)=\partial_\omega {\rm Con}_h(Z).
$$
An elementary calculation by using \eqref{esti-b} and \eqref{rem-eq-1} shows that for any $x$, $y\in {\rm Con}(Z)$,
$$
|(x|y)_o-(x|y)_b|\leq 2\log (2\diam Z),
$$
which shows that $\vartheta_{\varepsilon, b}$ is biLipschitz to $\vartheta_{\varepsilon, o}$. Theorem \ref{thm-2} recovers \cite[Theorem 8.1]{BSC} when $(Z, d_Z)$ is a complete bounded metric space.
 \end{rem}

\section{Proof of Theorem \ref{thm-3}}\label{sec-6}

Assume that $X$ is a visual Gromov $\delta$-hyperbolic  metric  space for some $\delta\geq 0$ and assume that $\omega\in\partial_G X$.
 Before going to proof  Theorem \ref{thm-3}, we first gives an illustration that Theorem \ref{thm-3} does not depend on the choices of visual metrics on $\partial_G X$ and $\partial_\omega X$.

If $d_1$, $d_2$ are two visual metrics on $\partial_G X$, then ${\rm Con}((\partial_G X, d_1))$ and ${\rm Con}((\partial_G X, d_2))$ are roughly similar by \cite[Theorem 7.4]{BSC}.
Similar arguments also ensure that ${\rm Con}((\partial_\omega X, d_1))$ and ${\rm Con}((\partial_\omega X, d_2))$ are roughly similar with visual metrics $d_1$, $d_2$ with respect to $\omega$ by using Theorem \ref{thm-1}.


Now, we define a visual metric on $\partial_G X$. The existence of such metric follows from \cite[Lemma 6.1]{BSC}: Fix a point $o\in X$, there is a $\varepsilon_0>0$ such that when $\varepsilon \delta\leq \varepsilon_0$, we can find a  visual metric $d_{o, \varepsilon}$ on $\partial_G X$ satisfying
 \begin{equation}\label{eq-m-d}
\frac{1}{2}e^{-\varepsilon(x|y)_o}\leq d_{o, \varepsilon}(x, y)\leq e^{-\varepsilon(x|y)_o}
\end{equation}
for all $x, y\in \partial_G X$.

In the following, we fix a $\varepsilon>0$ with $\varepsilon\delta\leq \varepsilon_0$, and let $d:=d_{o, \varepsilon}$.
Then we get a complete bounded metric space $(\partial_G X, d)$, see \cite[Proposition 6.2]{BSC}.

Next, we define a visual metric with respect to $\omega$ on $\partial_\omega X$. Consider the inversion $\rho$ of $d$ with respect to $\omega$, which is defined by
\begin{equation}\label{eq-rho}
\rho(x, y)=\frac{d(x, y)}{d(x, \omega)d(y, \omega)}
\end{equation}
for any $x, y\in\partial_\omega X$. Then $\rho$ is a quasimetric on $\partial_\omega X$.
There exists a metric $\rho_0$ on $\partial_\omega X$ such that $\rho$ is biLipschitz equivalent to $\rho_0$ by \cite[Lemma 3.1]{BHX}. More precisely, for any  $x, y\in \partial_\omega X$,
 \begin{equation}\label{eq-rho-0}
\frac{1}{4}\rho(x, y)\leq \rho_0(x, y)\leq \rho(x, y).
\end{equation}
See \cite[Section 3]{BHX} or \cite[Section 5.3]{BuSc} for more informations about the inversion of metric.

Combining \eqref{eq-m-d}, \eqref{eq-rho} and \eqref{eq-rho-0}, we see that there is a constant $c\geq1$ such that for any $x, y\in \partial_\omega X$,
\begin{equation}\label{eq-ve}
c^{-1}\frac{e^{-\varepsilon (x|y)_o}}{e^{-\varepsilon (x|\omega)_o-\varepsilon (y|\omega)_o}}\leq \rho_0(x, y)\leq c\frac{e^{-\varepsilon (x|y)_o}}{e^{-\varepsilon (x|\omega)_o-\varepsilon (y|\omega)_o}}.
\end{equation}

Let $b(\cdot)=b_{\omega, o}(\cdot)$, where $b_{\omega, o}(\cdot)$ is given in \eqref{Busemann}.
We infer from \cite[Example 3.2.1]{BuSc} that for any $u, v\in X$,
 $$
(u|v)_b=(u|v)_o-(u|\omega)_o-(v|\omega)_o,
 $$
and then, by using Lemma \ref{lem-fun-b} and \cite[Lemma 2.2.2]{BuSc}, for any $x, y\in\partial_\omega X$,
\begin{equation*}
(x|y)_b=(x|y)_o-(x|\omega)_o-(y|\omega)_o+O_\delta(1).
\end{equation*}
It follows from \eqref{eq-ve} that there is a $c_0\geq 1$ such that  for any $x, y\in \partial_\omega X$,
\begin{equation*}
c_0^{-1}e^{-\varepsilon (x|y)_b}\leq \rho_0(x, y) \leq c_0 e^{-\varepsilon (x|y)_b}.
\end{equation*}
This shows that $\rho_0$ is a visual metric on $\partial_\omega X$ based at $\omega$ with parameter
$\varepsilon$. In the following, we denote
\begin{equation*}
\partial_\omega X:=(\partial_\omega X, \rho_0)
\end{equation*}
and
\begin{equation*}
\partial_G X:=(\partial_G X, d).
\end{equation*}

\begin{proof}[{\bf Proof of Theorem \ref{thm-3}}]

Since $X$ is a visual Gromov hyperbolic metric space, $X$ is roughly similar to ${\rm Con}(\partial_G X)$ by \cite[Theorem 8.2]{BSC}. Recall that $\partial_G X$ is a complete bounded metric space, and $\partial_\omega X$ is a complete metric space.

By Theorem \ref{Thm-A}, the truncated hyperbolic filling
$$Y:={\rm Hyp}_T(\partial_G X)$$
is also a visual Gromov hyperbolic space. Moreover, there is an identification $\partial_G Y=\partial_G X$ of sets, and $d$ is biLipschtiz equivalent to $\alpha^{-(\cdot|\cdot)_o}$.
Here, $\alpha>10$ is one construction parameter of ${\rm Hyp}(\partial_G X)$.
Then $d$ is also a visual metric on $\partial_G Y$.  We equip $\partial_G Y$ with this metric $d$,  still denote this bounded metric space by
$$\partial_G Y:=(\partial_G Y,d).$$
Obviously, $\partial_G X$ is isometric to $\partial_G Y$, and thus, ${\rm Con}(\partial_G X)$ is roughly similar to ${\rm Con}(\partial_G Y)$ by \cite[Theorem 7.4]{BSC}.

By using \cite[Theorem 8.2]{BSC} again, $Y$ is roughly similar to ${\rm Con}(\partial_G Y)$, and then, is roughly similar to ${\rm Con}(\partial_G X)$.
As a result, $X$ is roughly similar to $Y$.

Recall that on $\partial_\omega X$, the visual metric $\rho_0$ is biLipschitz equivalent to the inversion
$$
\rho(x, y)=\frac{d(x, y)}{d(x, \omega)d(y, \omega)}.
$$
Using \cite[Theorem 7]{Jordi} and \cite[Remark 4]{Jordi}, we see that $Y={\rm Hyp}_T(\partial_G X)$ is roughly isometric to ${\rm Hyp}(\partial_\omega X)$, and hence, is roughly similar to ${\rm Con}_h(\partial_\omega X)$ by Theorem \ref{thm-2.9}.

Therefore,  $X$ is roughly similar to ${\rm Con}_h(\partial_\omega X)$.
\end{proof}

  \subsection*{Acknowledgments}
The first author (Manzi Huang) was partly supported by NNSF of China under the number 12371071.
The second author (Zhihao Xu) was partly supported by NNSF of China under the number 12071121 and Postgraduate Scientific Research Innovation Project of Hunan Province under the number CX20220506.

\noindent Manzi Huang,

\noindent
MOE-LCSM, School of Mathematics and Statistics, Hunan Normal University, Changsha, Hunan 410081, People's Republic of
China.

\noindent{\it E-mail address}:  \texttt{mzhuang@hunnu.edu.cn}\medskip\medskip

\noindent Zhihao Xu,

\noindent
MOE-LCSM, School of Mathematics and Statistics, Hunan Normal University, Changsha, Hunan 410081, People's Republic of
China.

\noindent{\it E-mail address}:  \texttt{zhxu@hunnu.edu.cn}


\vspace*{5mm}



\begin{thebibliography}{99}

	
	\bibitem{BBS} A.~Bj\"{o}rn, J.~Bj\"{o}rn, and N.~Shanmugalingam, {\it Extension and trace results for doubling metric measure spaces and their hyperbolic fillings}, J. Math. Pures Appl., 159 (2022), 196--249.

	
	\bibitem{BHK} M.~Bonk, J.~Heinonen, and P.~Koskela, {\it Uniformizing Gromov hyperbolic spaces}, Ast\'{e}risque No. 270 (2001), viii+99 pp.
	
	
	\bibitem{BSC} M.~Bonk and O.~Schramm, {\it Embeddings of Gromov hyperbolic spaces}, Geom. Funct. Anal., 10 (2000), 266--306.

\bibitem{BHX} S. M.~Buckley, D. A.~Herron and X.~Xie, {\it Metric space inversions, quasihyperbolic distance, and uniform spaces}, Indiana Univ. Math. J. 57 (2008), no. 2, 837--890.


\bibitem{BuSc} S.~Buyalo and V.~Schroeder, {\it Elements of Asymptotic Geometry}, EMS Monographs in Mathematics. European Mathematical Society (EMS), Z\"{u}rich, 2007. xii+200 pp.
	
	\bibitem{BuC} C.~Butler, {\it Uniformizing Gromov hyperbolic spaces with Busemann functions}, arXiv:2007.11143.
	

	\bibitem{GH} E.~Ghys and P.~de la Harpe, {\it Sur les groupes hyperboliques d'apr\`{e}s Mikhael Gromov}, Progr. Math., 83, Birkh\"{a}user Boston, Boston, MA, 1990.
	
	
	\bibitem{Gr87} M.~Gromov, {\it Hyperbolic groups}, Essays in group theory, 75--263, Math. Sci. Res. Inst. Publ., 8, Springer, New York, 1987.

	
	\bibitem{H} J.~Heinonen, {\it Lectures on analysis on metric spaces},  Universitext. Springer-Verlag, New York, 2001. x+140 pp.
	
\bibitem{Jordi} J.~Jordi, {\it Interplay between interior and boundary geometry in Gromov hyperbolic spaces}, Geom Dedicata 149 (2010), 129--154.
	
	\bibitem{TrV} D.~Trotsenko and J.~V\"{a}is\"{a}l\"{a}, {\it Upper sets and quasi-symmetric maps}, Ann. Acad. Sci. Fenn. Math., 24 (1999), 465--488.

	
	\bibitem{V} J.~V\"{a}is\"{a}l\"{a}, {\it Gromov hyperbolic spaces}, Expo. Math., 23 (2005), 187--231.

	
\end{thebibliography}
\end{document}